\documentclass[12pt]{article}
\usepackage{float}
\usepackage{makeidx}
\usepackage{latexsym}

\usepackage[margin=2.5cm]{geometry}
\usepackage[dvips]{graphicx}
\usepackage{amssymb}
\usepackage{amsmath}
\usepackage{amsthm}
\usepackage{lineno}

\newtheorem{theorem}{Theorem}[section]
\newtheorem{lemma}[theorem]{Lemma}
\newtheorem{corollary}[theorem]{Corollary}
\newtheorem{proposition}[theorem]{Proposition}
\newtheorem{conjecture}[theorem]{Conjecture}

\newtheorem{claim}[theorem]{Claim}
\theoremstyle{definition}

\begin{document}
\date{\today}
\title{Edge precoloring extension of hypercubes}

\author{
{\sl Carl Johan Casselgren}\footnote{Department of Mathematics, Link\"oping University, SE-581 83 Link\"oping, Sweden. 
{\it E-mail address:} carl.johan.casselgren@liu.se  }
\and {\sl Klas Markstr\"om}\footnote{Department of Mathematics, 
Ume\aa\enskip University, 
SE-901 87 Ume\aa, Sweden.
{\it E-mail address:} klas.markstrom@umu.se
}
\and {\sl Lan Anh Pham}\footnote{Department of Mathematics,
Ume\aa\enskip University, 
SE-901 87 Ume\aa, Sweden.
{\it E-mail address:} lan.pham@umu.se
}
}
\maketitle

\bigskip
\noindent
{\bf Abstract.}
We consider the problem of extending partial edge colorings
of hypercubes. In particular, we obtain an analogue of
the positive solution to the famous Evans' conjecture
on completing partial Latin squares by proving that every proper partial
edge coloring of at most $d-1$ edges of the $d$-dimensional 
hypercube $Q_d$ can be extended to a proper $d$-edge coloring
of $Q_d$. Additionally, we characterize which partial edge
colorings of $Q_d$ with precisely $d$ precolored edges
are extendable to proper $d$-edge colorings of $Q_d$.

\bigskip


\section{Introduction}

	An {\em edge precoloring} (or {\em partial edge coloring}) 
	of a graph $G$ is a proper edge coloring of some 
	subset $E' \subseteq E(G)$; {\em a $t$-edge precoloring}
	is such a coloring with $t$ colors.	
		An edge $t$-precoloring $\varphi$ is
	{\em extendable} if there is a proper $t$-edge coloring $f$
	such that $f(e) = \varphi(e)$ for any edge $e$ that is colored
	under $\varphi$; $f$ is called an {\em extension}
	of $\varphi$. 
	
	In general, the problem of extending a given edge precoloring
	is an $\mathcal{NP}$-complete problem, 
	already for $3$-regular bipartite graphs \cite{Fiala}.
	One of the earlier references explicitly discussing the problem of 
	extending a partial  edge coloring is \cite{MarcotteSeymour}; there a 
	simple necessary condition for the existence of an extension is 
	given and the authors  find a class  of graphs where this 
	condition is also sufficient.   More recently the question of 
	extending a precoloring where the precolored edges form 
	a matching has gathered interest. In \cite{EGHKPS} a 
	number of positive results and conjectures are given. 
	In particular it is conjectured that for every graph $G$,
	if $\varphi$ is an edge precoloring of a matching $M$ in $G$
	using $\Delta(G)+1$ colors,
	and any two edges in $M$ 
	are at distance at least $2$ from each other,  then $\varphi$ 
	can be extended to a proper $(\Delta(G)+1)$-edge coloring of $G$;
	this was first 
	conjectured in \cite{AlbersonMoore}, but then with 
	distance $3$ instead. 
	Here, as usual, $\Delta(G)$ denotes the maximum degree of a graph
	$G$, and
	by the {\em distance} between two
	edges $e$ and $e'$ we mean the number of edges in a 
	shortest path between an endpoint of $e$ and an endpoint of $e'$; 
	a {\em distance-$t$ matching} is a matching where any
	two edges are at distance at least $t$ from each other. A
	distance-$2$ matching is also called an {\em induced matching}.

	Note that the conjecture on distance-$2$ matchings in \cite{EGHKPS}
	is sharp both with respect to the
	distance between precolored edges,
	and in the sense that $\Delta(G)+1$ can in general 
	not be replaced by $\Delta(G)$,
	even if any two precolored edges are at arbitrarily 
	large distance from each other \cite{EGHKPS}.
	In \cite{EGHKPS}, it is proved that this conjecture holds
	for e.g. bipartite multigraphs and subcubic multigraphs, and
	in \cite{GiraoKang} it is proved that a version of the conjecture with the
	distance increased to 9 holds for general graphs.

	However, for one specific family of graphs, the balanced
	complete
	bipartite graphs $K_{n,n}$, the edge precoloring
	extension problem was studied far earlier than
	in the above-mentioned references. Here
	the  extension problem corresponds to asking whether a
	partial Latin square can be completed to a Latin square.
	In this form the problem appeared already in 1960, when Evans 
	\cite{Evans}  stated his now classic  conjecture  that for
	every positive integer $n$, if 
	$n-1$ edges in $K_{n,n}$ have been (properly) colored, 
	then the partial coloring can be extended to 
	a proper $n$-edge-coloring of $K_{n,n}$. 
	This conjecture was
	solved for large 
	$n$ by H\"aggkvist  \cite{Haggkvist78} and later for all $n$ by Smetaniuk 
	\cite{Smetaniuk}, 
	and independently by Andersen and Hilton \cite{AndersenHilton}.
	Moreover, Andersen and Hilton \cite{AndersenHilton}	characterized
	which $n \times n$ partial Latin squares with exactly $n$ non-empty
	cells are extendable.

	In this paper we consider the edge
	precoloring extension problem for the family of hypercubes. 
	Although matching extendability and subgraph containment problems
	have been studied extensively for hypercubes,
	(see e.g. \cite{VandenbusscheWest, Fink, WangZhang} and references
	therein)
	the edge precoloring extension problem for hypercubes
	seems to be a hitherto quite unexplored line of research.
	As in the setting of completing partial Latin squares (and unlike
	the papers \cite{EGHKPS,GiraoKang}) we consider only proper edge
	colorings of hypercubes $Q_d$ using exactly $\Delta(Q_d)$ colors.
	
	We prove that every edge precoloring of
	the $d$-dimensional hypercube $Q_d$ with at most
	$d-1$ precolored edges is extendable to a $d$-edge coloring
	of $Q_d$, thereby establishing an analogue of the positive
	resolution of Evans' conjecture.
	Moreover, similarly to \cite{AndersenHilton} we also characterize
	which proper precolorings with exactly $d$ precolored
	edges are not extendable to proper $d$-edge colorings of $Q_d$.
	We also consider the cases when the precolored
	edges form an induced matching, or one or two
	hypercubes of smaller dimension. The paper is concluded by a conjecture
	and some examples and remarks on edge precoloring extension of general
	$d$-regular bipartite graphs.

\section{Preliminaries}

	Unless otherwise stated all (partial) 
	edge colorings (or just {\em colorings})
	in this paper are proper. 
	Moreover, all proper $d$-edge colorings use colors $1,\dots, d$
	unless otherwise stated.
	If $\varphi$ is an edge precoloring of $G$,
	and an edge $e$ is colored under $\varphi$,
	then we say that $e$ is {\em $\varphi$-precolored}.

	If $\varphi$ is a (partial) proper $t$-edge coloring of $G$
	and $1 \leq a,b \leq t$, then a path or cycle in
	$G$ is called {\em $(a,b)$-colored under $\varphi$} if
	its edges are colored by colors $a$ and $b$ alternately.
	
	In the above definitions, we often leave out the explicit 
	reference to a
	coloring $\varphi$, if the coloring
	is clear from the context.
	
	\bigskip
	
	Havel and Moravek \cite{HavelMoravek} (see also \cite{Harary})
	proved a criterion
	for a graph $G$ to be a subgraph of a hypercube:

\begin{proposition}
\label{prop:HavelMoravek}
	A graph $G$ is a subgraph of $Q_d$ if and only if 
	there is a proper $d$-edge coloring of 
	$G$ with integers $\{1,\dots,d\}$
	such that
	\begin{itemize}
	%
	\item[(i)] in every path of $G$ there is some color that appears an odd
	number of times;
	
	\item[(ii)] in every cycle of $G$ no color appears an odd number of times.
	\end{itemize}
\end{proposition}

	A {\em dimensional matching} $M$ of $Q_d$ is a perfect matching of $Q_d$
	such that $Q_d- M$ is isomorphic to two copies of $Q_{d-1}$; evidently
	there are precisely $d$ dimensional matchings in $Q_d$.
	We shall need the following easy lemma.
	
	\begin{lemma}
	\label{lem:matchings}
		Let $d \geq 2$ be an integer.
		Then there are $d$ different dimensional matchings in $Q_d$;
		indeed $Q_d$ decomposes into $d$ such perfect matchings.
	\end{lemma}
	The proof is left to the reader. 
	
	Intuitively, the colors in
	the proper edge coloring in Proposition \ref{prop:HavelMoravek}
	correspond to dimensional matchings in $Q_d$ (as pointed out in \cite{Harary}).
	In particular, Proposition \ref{prop:HavelMoravek} holds if we take the
	dimensional matchings as the colors. Furthermore we have the following.

	\begin{lemma}
	\label{lem:DimMathInduce}
		The subgraph induced by $r$ dimensional matchings
		in $Q_d$ is isomorphic to a disjoint union of
		$r$-dimensional hypercubes.
	\end{lemma}
	
	This simple observation shall be used quite frequently below.
	
	We shall also need some standard definitions
	on list edge coloring.
	Given a graph $G$, assign to each edge $e$ of $G$ a set
	$\mathcal{L}(e)$ of colors.
	Such an assignment $\mathcal{L}$ is called
	a \emph{list assignment} for $G$ and
	the sets $\mathcal{L}(e)$ are referred
	to as \emph{lists} or \emph{color lists}.
	If all lists have equal size $k$, then $\mathcal{L}$
	is called a \emph{$k$-list assignment}.
	Usually, we seek a proper
	edge coloring $\varphi$ of $G$,
	such that $\varphi(e) \in \mathcal{L}(e)$ for all
	$e \in E(G)$. If such a coloring $\varphi$ exists then
	$G$ is \emph{$\mathcal{L}$-colorable} and $\varphi$
	is called an \emph{$\mathcal{L}$-coloring}. 
	Denote by $\chi'_L(G)$ the minimum integer $t$
	such that $G$ is $\mathcal{L}$-colorable
	whenever $\mathcal{L}$ is a $t$-list assignment.
	A fundamental result in list edge coloring theory
	is the following theorem by Galvin
	\cite{Galvin}. As usual, $\chi'(G)$ denotes the chromatic
	index of a multigraph $G$.

\begin{theorem}
\label{th:Galvin}
	For any bipartite multigraph $G$, $\chi_L'(G) = \chi'(G)$.
\end{theorem}

\section{Extending edge precolorings of hypercubes}

We begin this section by giving a short proof of the
following theorem, thereby establishing an analogue
for hypercubes to the positive solution of the Evans' conjecture.

	\begin{theorem}
	\label{th:hypercube}
		Let $d \geq 2$ be a positive integer.
		If $\varphi$ is an edge precoloring of
		at most $d-1$ edges of the hypercube $Q_d$,
		then $\varphi$ can be extended to a proper $d$-edge coloring
		of $Q_d$.
	\end{theorem}
	\begin{proof}
		The proof is by induction on $d$. For $d=2$, the statement
		is straightforward.
		
		Suppose that $d > 2$ and that the theorem holds for $Q_{d-1}$.
		Let $\varphi$ be an edge precoloring of at most $d-1$ edges
		of $Q_d$.
		By Lemma \ref{lem:matchings}, $Q_d$ has $d$ perfect matchings
		$M$ such that $Q_d - M$ is the disjoint union of two copies
		of $Q_{d-1}$. Since at most $d-1$ edges of $Q_d$ are precolored,
		there is such a perfect matching $\hat M$ satisfying that
		no edge of $\hat M$ is precolored. 
		Let $H_1$ and $H_2$ be the components of
		$Q_d - \hat M$. We distinguish between two different cases.
		
		\bigskip
		
		\noindent
		{\bf Case 1.} {\em $H_1$ has at least $1$ and at most $d-2$
		precolored edges:}
		
		\medskip
		
		\noindent
		Without loss of generality we assume that 
		the precoloring of $Q_d$ uses colors $1,\dots, d-1$.
		Since $H_1$
		contains  at most $d-2$ precolored edges, there is,
		by the induction hypothesis, a proper 
		$(d-1)$-edge coloring $\varphi_1$ of $H_1$ which 
		is an extension of the restriction of $\varphi$ to $H_1$.
		Similarly, there is a proper $(d-1)$-edge coloring $\varphi_2$ of
		$H_2$ which is an extension of the restriction
		of $\varphi$ to $H_2$. 
		By coloring the edges of $\hat M$ with color $d$,
		we obtain a proper $d$-edge coloring of $Q_d$.

		\bigskip

		\noindent 
		{\bf Case 2.} {\em $H_1$ has exactly $d-1$ precolored edges:}
		
		\medskip
		
		\noindent
		Without loss of generality we assume that at least one edge in
		$Q_d$ is precolored with color $1$.
		Define a new edge precoloring $\varphi'$ of $Q_d$ by removing color $1$
		from any precolored edge of $Q_d$ that is colored $1$. 
		By the induction hypothesis,
		there is a proper $(d-1)$-edge coloring $\varphi'_1$
		of $H_1$
		using colors $2,3,\dots, d$ which is an extension of $\varphi'$.
		From $\varphi'_1$ we define a new proper edge coloring $\varphi_1$
		of $H_1$
		by setting $\varphi_1(e)=1$ for every edge $e$ with $\varphi(e)=1$,
		and retaining the color of every other edge of $H_1$.
		Then $\varphi_1$ is an extension of $\varphi$ on the graph
		$H_1$.
		
		Let $\varphi_2$ be an edge coloring of $H_2$ obtained
		by coloring every edge of $H_2$
		with the color of  the corresponding edge of $H_1$ 
		under $\varphi_1$.\footnote{Here, and in the following, two edges of
		$H_1$ and $H_2$ are {\em corresponding} if their endpoints are joined
		by two edges of $M$. Similarly, two vertices are {\em corresponding}
		if they are joined by an edge of $M$.}
		Now, for any vertex $v$ of $H_1$, if color
		$t$ does not appear on an edge incident
		to $v$, $1 \leq t \leq d$, then color $t$ does not appear on any
		edge incident to the corresponding vertex of $H_2$.
		Thus we may extend $\varphi_1$ and $\varphi_2$ to a proper edge
		coloring $\psi$ of $Q_d$ by, for any edge $e$ of $\hat M$, coloring
		$e$ with the color in $\{1,2, \dots, d\}$ 
		not appearing on any edge incident to 
		one of its endpoints. Clearly, $\psi$ is an extension of $\varphi$.

		\bigskip
		By symmetry, it suffices to consider the two different cases above.
		Hence, the theorem follows.	
	\end{proof}

	Ryser \cite{Ryser} proved a necessary and sufficient condition for
	an $n \times n$ partial Latin square where all non-empty cells
	lie in a completely filled $r \times s$ subrectangle
	to be completable. In particular, his result implies that
	any $n \times n$ partial Latin square, where all non-empty
	cells lie within an $\lfloor n/2 \rfloor \times \lfloor n/2 \rfloor$
	subarray is completable.
	We note the following analogue for hypercubes:
	
	\begin{proposition}
	\label{prop:MHallhypercube}
		If $\varphi$ is a proper $d$-edge coloring of $Q_r \subseteq Q_d$,
		then $\varphi$ can be extended to a proper edge coloring
		of $Q_d$.
	\end{proposition}
	 We provide a brief sketch of the proof.
	\begin{proof}[Proof (sketch).]
	Evidently, $Q_r$ is a component of the subgraph of $Q_d$ induced by exactly
	$r$ dimensional matchings in $Q_d$. It suffices to prove that
	if $Q_{r+1}$ is a hypercube of dimension $r+1$ which is contained
	in $Q_d$, and which contains $Q_r$, then there is a proper $d$-edge 
	coloring of $Q_d$ that agrees with $\varphi$. 
	However, such a graph $Q_{r+1}$ consists of
	two copies of $Q_r$ and a dimensional matching joining corresponding
	vertices of the two copies of $Q_r$. We may thus obtain a proper
	$d$-edge coloring of $Q_{r+1}$ as in the proof of the preceding
	theorem. 
	\end{proof}
	If we do not insist that all edges in a subgraph of $Q_d$ isomorphic to
	$Q_r$ have to be precolored, then we have the following.
	
	\begin{corollary}
	\label{cor:Ryser}
		If $r \leq d/2$, then
		any partial proper edge coloring of 
		$Q_r \subseteq Q_d$ 
		with colors $1,\dots,d$
		can be extended to a proper $d$-edge coloring of $Q_d$.
	\end{corollary}
	
	\begin{proof}
	It suffices to prove that there is a proper $d$-edge coloring
	of $Q_r$ that agrees with the given partial edge coloring $\varphi$
	of $Q_r$; invoking Proposition \ref{prop:MHallhypercube} then yields
	the desired result. Since $r \leq d/2$, such a proper $d$-edge coloring
	can be obtained by greedily coloring the uncolored edges of $Q_r$.
	\end{proof}

	Note that the bound on $r$ is sharp, since there is a partial proper
	edge coloring of
	$Q_{d/2+1}$ with colors $1,\dots, d$ that cannot be extended to a proper
	$d$-edge coloring of $Q_d$: Let $uv$ be an edge of $Q_{d/2+1}$
	and color the edges incident with $u$ and distinct from $uv$ by colors
	$1,\dots, d/2$, respectively; 
	color the edges incident with $v$ and distinct from
	$uv$ by colors $d/2+1, \dots, d$, respectively. The resulting partial
	edge coloring can clearly not be extended to a proper $d$-edge coloring
	of $Q_d$.

	Our next result establishes an analogue for hypercubes of the characterization
	of Browning et al. \cite{BrowVojtWanless}
	of when a partial Latin square, the non-empty cells of which
	constitute two Latin subsquares,
	is completable.
	
	\begin{theorem}
	\label{th:TwoCubes}
		Let $Q_{k_1}$ and $O_{k_2}$ be two hypercubes
		of dimension $k_1$ and $k_2$, respectively,
		contained in
		a $d$-dimensional hypercube $Q_d$, and let $f$ be a proper
		edge coloring of $Q_{k_1} \cup O_{k_2}$ such that the restriction of
		$f$ to $Q_{k_1}$ $(O_{k_2})$
		is a proper edge coloring using $k_1$ $(k_2)$
		colors $A_1$ $(A_2)$ from  
		$\{1,\dots,d\}$.
		Then the coloring $f$ is extendable to a proper $d$-edge coloring
		of $Q_d$ unless $Q_{k_1}$ and $O_{k_2}$ are disjoint,
		a vertex of $Q_{k_1}$ is adjacent to a vertex of $O_{k_2}$,
		and $d \leq |A_1 \cup A_2|$.
		\end{theorem}
	
	We shall need the following easy lemma; the proof is left to the reader.

	\begin{lemma}
	\label{lem:intersect}
		Let $Q_{k_1}$ and $O_{k_2}$ be hypercubes contained in a hypercube
		$Q_d$ of larger dimension. If $Q_{k_1} \cap O_{k_2} \neq \emptyset$,
		then the intersection $Q_{k_1} \cap O_{k_2}$ is a hypercube of a smaller
		dimension.
	\end{lemma}

	\begin{proof}[Proof of Theorem \ref{th:TwoCubes}]
		Let $f_1$ ($f_2$) denote
		the restriction of the coloring $f$ to $Q_{k_1}$ ($O_{k_2}$).
		Let $\mathcal{M}$ be the set of dimensional matchings in $Q_d$, and
		denote by $\mathcal{M}_1$ and $\mathcal{M}_2$ 
		the set of dimensional matchings 
		that $Q_{k_1}$ and $O_{k_2}$ occupies, respectively. 
		Assume that $Q_{k_1}$ and $O_{k_2}$ together 
		contains edges from $k$ dimensional matchings, 
		put $\mathcal{M}_k=\mathcal{M}_1 \cup \mathcal{M}_2$, and
		let $\mathcal{Q}_k$
		be the set of subhypercubes of $Q_d$ 
		induced by all the dimensional matchings in $\mathcal{M}_k$.
		
		Let $H_1$ and $H_2$ be the components of $\mathcal{Q}_k$
		that contains $Q_{k_1}$ and $O_{k_2}$, respectively.
		Suppose first that $Q_{k_1}$ and $O_{k_2}$ are disjoint
		subgraphs of $Q_d$. This implies that $H_1$ and $H_2$ are disjoint.
		
		By Proposition \ref{prop:MHallhypercube}, there is a proper
		edge coloring $g_1$ of $H_1$ which agrees with $f_1$ and
		uses exactly $k$ colors from $\{1,\dots,d\}$, and
		a proper
		edge coloring $g_2$ of $H_2$ which agrees with $f_2$ and uses exactly
		$k$ colors from $\{1,\dots,d\}$ 
		(possibly distinct from the ones used in the coloring of
		$H_1$). Additionally, we choose
		these edge colorings so that $g_i$ uses as many colors from $A_{3-i}$
		as possible. 
		
		Note that if the coloring $g_1$ or $g_2$
		uses some color not in $A_1 \cup A_2$, then
		$|A_1 \cup A_2| < k$,
		and both $g_1$ and $g_2$ uses all colors in $A_1 \cup A_2$ and
		$k-|A_1 \cup A_2|$ additional colors from $\{1,\dots,d\}$. Clearly,
		we may then assume that $g_1$ and $g_2$ uses the same additional colors
		from $\{1,\dots,d\} \setminus (A_1 \cup A_2)$.

		\bigskip
		
		\noindent
		{\bf Case 1.} {\em There is an edge $e$ between a vertex of $H_1$
		and a vertex of $H_2$:}
		
		We prove that the coloring $f$ can be extended to a $d$-edge coloring
		of $Q_d$ if $d - |A_1 \cup A_2| > 0$.
		
		Let $M$ be the dimensional matching that 
		contains $e$. Consider 
		the set of subhypercubes $\mathcal{Q}_{k+1}$ 
		induced by the set of dimensional matchings 
		$\mathcal{M}_k \cup \{M\}$.
		Since $e$ is adjacent to both vertices of $H_1$ and $H_2$
		we have that $H_1$ and $H_2$ are subgraphs of the
		same component $H$ in $\mathcal{Q}_{k+1}$.
		
		Now, if $|A_1 \cup A_2|<k$, then $g_1$ and $g_2$ uses the same $k$ 
		colors from $\{1,\dots,d\}$. Moreover, $d \geq k+1$, because
		$M \notin \mathcal{M}_k$. This implies that 
		there is a color
		$c \in \{1,\dots,d\}$ which is not used in the coloring $g_1$ or $g_2$. 
		By coloring all the edges of the dimensional matching
		$M$ with one endpoint in $H_1$ and one endpoint in $H_2$ by color $c$, 
		we obtain a proper edge coloring of $H$; by Proposition 
		\ref{prop:MHallhypercube} this edge coloring can be
		extended to a proper $d$-edge coloring of $Q_d$.
		Clearly, this coloring is an extension of $f$.

		If, on the other hand, $|A_1 \cup A_2| \geq k$, then $g_1$ and $g_2$
		uses only colors from $A_1 \cup A_2$, and since $d > |A_1 \cup A_2|$,
		there is a color
		$c \in \{1,\dots,d\}$ which is not used in the coloring $g_1$ or $g_2$;
		as in the preceding paragraph, we conclude that $f$ is extendable.
		
		\bigskip
		
		\noindent
		{\bf Case 2.} {\em  There is no pair of adjacent vertices where one 
		is in $H_1$ and the other in $H_2$:}
		
		Consider the graph $\mathcal{Q}_k$; 
		by Lemma \ref{lem:DimMathInduce}, $\mathcal{Q}_k$
		consists of disjoint $k$-dimensional hypercubes.
		We define a new graph $G$ 
		where every component $H_i$ in 
		$\mathcal{Q}_k$ is represented by a vertex $u_{H_i}$,
		and where
		$u_{H_i}$ and $u_{H_j}$, $i \neq j$, 
		are adjacent if there is an edge joining
		a vertex of $H_i$ with a vertex of $H_j$. 
		It is easy to see that $G$ is a regular
		bipartite graph with degree $d-k$.
		
		We define a list assignment $\mathcal{L}$ for $G$
		by for every edge $e=u_{H_i}u_{H_j}$ of $G$
		and every color $c \in \{1,\dots,d\}$ including $c$
		in $\mathcal{L}(e)$ if
		\begin{itemize}
		\item $c$ does not appear in the coloring of $H_1$ if $i=1$ or $j=1$.
		\item $c$ does not appear in the coloring of $H_2$ if $i=2$ or $j=2$.
		\end{itemize}
		Since $H_1$ and $H_2$ do not contain pairs of adjacent vertices,
		$|\mathcal{L}(e)| \geq  d - k$ for all edges $e \in E(G)$.
		Thus, by Theorem \ref{th:Galvin},
		there is a proper edge coloring of $G$ with support in the lists.
		By coloring all edges going between $H_i$ and $H_j$ by 
		the color of the edge
		$e=u_{H_i}u_{H_j}$, and coloring every uncolored 
		subhypercube $H_i$ in $\mathcal{Q}_k$
		by $k$ colors which does not appear on the edges incident
		with $u_{H_i}$ in $G$, we obtain a proper $d$-edge coloring of $Q_d$
		that is an extension of $f$.

		\bigskip
		
		Let us now consider the case when $Q_{k_1}$ and $O_{k_2}$ 
		are not disjoint.
		If $Q_{k_1}$ and $O_{k_2}$ intersect in only one vertex, then
		$Q_{k_1}$ and $O_{k_2}$ occupy different dimensional matchings
		and $A_1 \cap A_2 = \emptyset$. Hence, for $i=1,2$,
		by Lemma \ref{lem:DimMathInduce}
		and K\"onig's edge coloring theorem,
		there is a proper
		edge coloring $g_i$ with colors only from $A_i$
		of the subgraph of $Q_d$ induced by the matchings
		in $\mathcal{M}_i$ which agrees with $f_i$.
		Similarly, the subgraph of $Q_d$ induced by 
		$\mathcal{M} \setminus \mathcal{M}_k$
		is $(d-k)$-regular; so if $d>k$,
		then there is, by K\"onig's edge coloring theorem,
		a proper $(d-k)$-edge coloring of this graph using colors only
		from the set $\{1,\dots,d\}\setminus (A_1 \cup A_2)$.
		This coloring, along with $g_1$ and $g_2$, yields a proper $d$-edge coloring
		of $Q_d$ that is an extension of $f$.
	
		\bigskip
		
		Suppose now that $Q_{k_1} \cap O_{k_2}$ contains at least one edge;
		by Lemma \ref{lem:intersect},
		this intersection is an $r$-dimensional hypercube $D_r$
		($r\geq 1$). Also, $H_1 = H_2$. 
		
		We shall prove that there is a proper edge coloring of $H_1$ that agrees
		with $f$ and uses at most $d$ colors; 
		the result then follows by invoking 
		Proposition \ref{prop:MHallhypercube}.
		If $D_r = O_{k_2}$ (or $D_r = Q_{k_1}$), then
		obviously $f$ is extendable, so we assume that this is not
		the case. Thus $k_2-r \geq 1$.
		
		Let us consider the restriction $f_r$ of the coloring $f$
		to $D_r$. Since $Q_{k_1}$ and $O_{k_2}$ 
		are both regular bipartite graphs,
		and the restriction of $f$ to $Q_{k_1}$ and
		$O_{k_2}$ are both proper edge colorings using
		a minimum number of colors, the coloring $f_r$ is a proper edge coloring
		using exactly $r$ colors; that is, $|A_1 \cap A_2| =r$.
		
		Consider the subgraph $\mathcal{Q}_{k_1}$ 
		of $Q_d$ induced by all dimensional matchings in
		$\mathcal{M}_1$.
		Consider a subhypercube $Q'_{k_1}$ of dimension $k_1$
		in $\mathcal{Q}_{k_1}$ that lies in $H_1$,
		and such that the vertices of $Q_{k_1}$ and $Q'_{k_1}$ 
		are adjacent via a subset $M_1$ of edges lying in a dimensional matching.
		Note that some edges of $M_1$ and $Q'_{k_1}$ 
		are in $O_{k_2}$. Let $S_1 = E(Q'_{k_1}) \cap E(O_{k_2})$, 
		$T_1=M_1 \cap E(O_{k_2})$.
		By coloring the edges of 
		$E(Q'_{k_1}) \setminus S_1$ by the colors of the corresponding edges in 
		$Q_{k_1}$ and coloring all the edges of $M_1 \setminus T_1$ by 
		a fixed color $c \in A_2 \setminus A_1$ (such a color exists
		since $k_2-r \geq 1$),
		we obtain an edge coloring of the subhypercube $Q_{k_1+1}$
		containing $Q_{k_1}$ and $Q'_{k_1}$. This edge coloring
		is proper, since all common colors of $A_1$ and $A_2$ appears in
		the coloring of $D_r$ and are therefore not used in the coloring
		of $E(Q'_{k_1}) \setminus S_1$.
		Moreover, 
		$O_{k_2} \cap Q_{k_1+1}$ is an $(r+1)$-dimensional hypercube
		$D_{r+1}$ containing $D_r$, and if $u$ is an arbitrary vertex of 
		$D_{r+1}$, then the set of colors incident with $u$ in 
		$Q_{k_1+1} - E(D_{r+1})$ is disjoint from $A_2$. 
		
		If $k_2 -r =1$, then we are done; the constructed edge coloring
		of $H_1$ can by Proposition \ref{prop:MHallhypercube} be extended
		to a proper $d$-edge coloring of $Q_d$.
		
		Suppose now that $k_2 -r \geq 2$.
		Let $A_{k_1 +1}$ be the set of
		colors in $A_2$ that has not been used in the coloring of
		$Q_{k_1 +1} - E(D_{r+1})$;
		since the coloring of $Q_{k_1+1} - E(D_{r+1})$ is a
		proper $(k_1+1)$-edge coloring
		in which $k_1$ colors are in $A_1$,
		we have $|A_{k_1 +1}| = k_2-r -1\geq 1$.
		Consider a subhypercube $Q'_{k_1 +1}$
		of $H_1$ that occupy the same
		dimensional matchings as the subhypercube $Q_{k_1 +1}$, and
		such that the vertices of $Q_{k_1 +1}$
		and $Q'_{k_1+1}$ are adjacent via a subset $M_2$ of edges lying
		in a dimensional matching.
		 Note that some edges of $M_2$ and $Q'_{k_1+1}$ 
		are in $O_{k_2}$.
		Let $S_2=E(Q'_{k_1+1}) \cap E(O_{k_2})$, $T_2=M_2 \cap E(O_{k_2})$.
		By coloring the edges of $E(Q'_{k_1 +1}) \setminus S_2$ by the colors of 
		corresponding edges in
		$Q_{k_1 +1}$ and coloring all the edges of
		$M_2 \setminus T_2$ by a fixed color $c \in A_{k_1+1}$,
		we obtain a proper edge coloring of the subhypercube $Q_{k_1+2}$ 
		containing $Q_{k_1 +1}$ and $Q'_{k_1+1}$, and where
		$O_{k_2} \cap Q_{k_1+2}$ is an $(r+2)$-dimensional hypercube
		$D_{r+2}$ containing $D_{r+1}$.
		Moreover, if $u$ is an arbitrary vertex of $D_{r+2}$, 
		then the set of colors incident with $u$
		in $Q_{k_1+2} - E(D_{r + 2})$ is disjoint from $A_2$. 
		
		Now, if $k_2-r =2$, then we are done; 
		otherwise, we continue the above process until we get 
		a proper edge coloring of $H_1$, which can then be
		extended to a proper edge coloring of $Q_d$ by  
		Proposition \ref{prop:MHallhypercube}.		
	\end{proof}

	\bigskip
	
	Next, we consider the case when all precolored
	edges lie in a matching. We would like to propose the following:
	
	\begin{conjecture}
	\label{conj:inducedmatch}
			If $\varphi$ is an edge precoloring of $Q_d$ where
			all precolored edges lie in an induced matching,
			then $\varphi$ is extendable to a proper $d$-edge coloring.
	\end{conjecture}
	
	In \cite{CasselgrenMarkstromPham}, we proved that this conjecture
	is true under the stronger assumption that every precolored edge
	is of distance at least $3$ from any other precolored edge.
	Moreover, by results in \cite{VandenbusscheWest}, Conjecture
	\ref{conj:inducedmatch} is true in the case when all precolored
	edges have the same color.
	
	Here we prove that the conjecture is true when all precolored edges lie
	in at most two distinct dimensional matchings.
	
	\begin{proposition}
		If the precolored edges of $Q_d$ form an induced matching
		all edges of which lie in two dimensional matchings,
		then the precoloring is extendable.
	\end{proposition}
	\begin{proof}
		Let $M_1$ and $M_2$ be the two dimensional matchings of 
		$Q_d$ containing all
		precolored edges. 
		Denote this precoloring by $\varphi$.
		By Lemma \ref{lem:DimMathInduce},
		$Q_d - M_1 \cup M_2$ is isomorphic to four copies
		$H_1, \dots, H_4$ of the $(d-2)$-dimensional
		hypercube. 
		Moreover, the graph $Q_d[M_1 \cup M_2]$ induced by 
		$M_1 \cup M_2$ is a disjoint union
		of $2$-dimensional hypercubes, and every vertex of $H_i$ is adjacent
		to precisely two edges from $Q_d[M_1 \cup M_2]$.
		
		Since the precolored edges form an induced matching, at most
		one edge of each component of $Q_d[M_1 \cup M_2]$ is precolored.
		From the precoloring $\varphi$ of $Q_d[M_1 \cup M_2]$ we define
		an edge precoloring $\varphi'$ of $Q_d[M_1 \cup M_2]$ 
		that satisfies the following:
		
		\begin{itemize}
		
			\item $\varphi'$ agrees with $\varphi$ on any edge that is colored
			under $\varphi$;
			
			\item for each component of $Q_d[M_1 \cup M_2]$, exactly two
			edges in this component are colored under $\varphi'$; moreover,
			these two edges are non-adjacent and have the same color under $\varphi'$.
		
		\end{itemize}
		Trivially, there is such a precoloring $\varphi'$; so
		to prove the theorem, it suffices to prove that there
		is a proper $d$-edge coloring $f$ of $H_1$ such that 
		for every edge $e$ of $H_1$, there is no adjacent edge $e'$ in 
		$Q_d[M_1 \cup M_2]$ such that $f(e) = \varphi'(e')$.
		This follows from the observation that given such a 
		coloring $f$ of $H_1$, we may color
		the edges of $H_2, H_3$ and $H_4$ correspondingly, and thereafter
		color the uncolored edges of $Q_d[M_1 \cup M_2]$ by for each edge using
		the unique color not appearing at any of its endpoints.
		
		To construct such a coloring of the edges of $H_1$
		we define a list assignment $L$ for $H_1$ by
		for every edge $e \in E(H_1)$ setting
		$$L(e) = \{1,\dots, d\} \setminus 
		\{\varphi'(e') : \text{ $e'$ is adjacent to $e$} \}.$$ Since
		every edge of $H_1$ is adjacent to two $\varphi'$-precolored edges, 
		$|L(e)| \geq d-2$ for every edge $e \in E(H_1)$. Hence,
		by Theorem \ref{th:Galvin}, 
		there is an $L$-coloring of $H_1$.
	\end{proof}

	Note that the condition on the matching being induced is best possible
	in terms of size of a precolored subset of a dimensional matching that is
	extendable to a proper $d$-edge coloring of $Q_d$. To see this,
	color all $2^{d-2}$ edges of a maximal induced matching $M_1$ contained in
	a dimensional matching $M$ with color $1$.
	Note that any extension of this precoloring uses color $1$ on all edges
	of $M$, because $M_1$ is a maximal induced matching of $M$. 
	So by coloring one edge of $M \setminus M_1$
	by color $2$, we obtain a non-extendable edge precoloring.

	\bigskip
	
	Next, we shall establish an analogue for hypercubes
	of the characterization by Andersen and Hilton \cite{AndersenHilton}
	of which $n \times n$ partial Latin squares with exactly $n$
	nonempty cells are completable.
	We shall prove that a proper precoloring of at most $d$ edges in $Q_d$ 
	is always extendable unless
	the precoloring $\varphi$ satisfies any 
	of the following conditions:

	\begin{itemize}
	 
		 \item[(C1)]  there is an uncolored edge $uv$ in $Q_d$ such that
		$u$ is incident with edges of $k \leq d$ distinct colors and
		$v$ is incident to $d-k$ edges colored with $d-k$
		other distinct colors (so $uv$
		is adjacent to edges of $d$ distinct colors);
		
		\item[(C2)]  there is a vertex $u$ 
		and a color $c$ such that $u$ is incident with at least one colored
		edge, $u$ is not incident with any edge of color $c$, and every
		uncolored edge incident with $u$ is adjacent to another edge colored $c$;

		\item[(C3)] there is a vertex $u$ and a color $c$
		such that every edge
		incident with $u$ is uncolored
		and every edge incident with $u$ is adjacent to another edge colored $c$;

		\item[(C4)]
		$d=3$ and the three precolored edges use three different colors and
		form a subset of a dimensional matching.

	\end{itemize}
		For $i=1,2,3,4$, we denote by $\mathcal{C}_i$
	the set of all colorings of $Q_d$, $d \geq 1$,
	satisfying the corresponding condition above, and we set
	$\mathcal{C} = \cup \mathcal{C}_i$.
	Let us briefly verify that if $\varphi$ is a precoloring of $Q_d$ 
	with exactly $d$ precolored
	edges and $\varphi \in \mathcal{C}$, then $\varphi$ is not extendable.
	
	Suppose first that the precoloring $\varphi$ satisfies condition (C1).
	Since the edge $uv$ is adjacent to edges of $d$ distinct colors,
	there is no proper $d$-edge coloring of $Q_d$ that agrees with $\varphi$.
	If $\varphi$, on the other hand, satisfies condition (C2), then since
	$u$ has degree $d$, any extension of $\varphi$ satisfies that the
	color $c$ must appear on one of the edges in
	$\{uv_1, \dots, uv_k\}$. However, such a $d$-edge coloring cannot be proper
	since this implies that there is a vertex that is incident with two edges
	colored $c$.
	
	Suppose now that $\varphi$ satisfies condition (C3). If $f$ is an extension
	of $\varphi$, then since $u$ has degree $d$, at least one edge incident with
	$u$ is colored $c$. However, such a $d$-edge coloring is not proper, so
	$\varphi$ is not extendable. That $\varphi$ is not extendable if it satisfies
	condition (C4) is a straightforward verification and is left to the reader.

	\begin{theorem}
	\label{th:dprecol}
		If $\varphi$ is a proper $d$-edge precoloring of $Q_d$ 
		with exactly $d$ precolored edges
		and $\varphi \notin \mathcal{C}$, then $\varphi$ is extendable to a proper
		$d$-edge coloring of $Q_d$.
	\end{theorem}
	
	The proof of this theorem is rather lengthy so we devote the next section
	to this proof.


	\section{Proof of Theorem \ref{th:dprecol}}

	The proof of Theorem \ref{th:dprecol} proceeds by induction.
	It is easily seen that the theorem holds
	when $d \in \{1,2\}$; let us consider the case when
	$d=3$.
	
	Let $\varphi$ be a precoloring of $Q_3$ and let us first assume
	that all precolored edges have the same color.
	If all three precolored edges lie in distinct dimensional matchings,
	then $\varphi \in \mathcal{C}_3$, and if all three edges lie in the
	same dimensional matching, then 
	we may color all the edges in this dimensional matching by the same color,
	and then obtain an extension of
	$\varphi$ by K\"onig's edge coloring theorem.
	Moreover, in the case when exactly two of the precolored edges
	are in the same dimensional matching, then these two
	edges must be at distance $1$ from each other, and so there is a perfect
	matching containing all precolored edges; hence, $\varphi$ is extendable.
		
		Suppose now that two colors appear on the precolored edges.
		Let $e_1,e_2,e_3$ be the precolored edges of $Q_3$
		and assume that two edges from $\{e_1,e_2,e_3\}$, say $e_1$ and $e_2$,
		have the same color and $e_3$ has another color under $\varphi$.
		If $e_1$ and  $e_2$ lie in the same dimensional matching, then
		$\varphi$ is extendable provided that there is a perfect matching
		of $Q_3$ containing $e_1$ and $e_2$, but not $e_3$.
		If $e_1$ and $e_2$ lie on a common $4$-cycle, then there is certainly 
		such a matching;
		if $e_1$ and $e_2$ does not lie on a common $4$-cycle, then
		this holds unless
		$\varphi \in \mathcal{C}_2$.

		Let us now
		assume that $e_1$ and $e_2$
		lie in different dimensional matchings. By symmetry,
		we may assume that $e_1$
		is any fixed edge of $Q_3$, which then yields
		$4$ different choices for the edge $e_2$, because every edge of $Q_3$
		is adjacent to exactly four other edges. 
		In fact, again by symmetry, it suffices to consider the two 
		different cases when $e_2$ is in different dimensional matchings
		(distinct from the one containing $e_1$).
		It is straightforward
		to verify
		that in both cases, the edges $e_1$ and $e_2$ are contained
		in a perfect
		matching not containing $e_3$ unless
		$\varphi \in \mathcal{C}_2$. Hence, if $\varphi \notin \mathcal{C}$,
		then $\varphi$ is extendable.
		
		Finally, let us consider the case when 
		three distinct colors appear on edges
		under $\varphi$. If all three precolored 
		edges $e_1, e_2, e_3$ lie in distinct dimensional matchings,
		then $\varphi$ trivially is extendable. Moreover, 
		since $\varphi \notin \mathcal{C}$, all three precolored edges do not
		lie in the same dimensional matching. Hence, it suffices to consider
		the case when exactly two of the precolored edges lie in the same
		dimensional matching. We assume $\varphi(e_i) = i$.
		
		Suppose, without loss of generality, that $e_1$ and $e_2$ lie in
		the same dimensional matching. We first consider the case when
		$e_1$ and $e_2$ lie on a common $4$-cycle.
		Since $\varphi \notin \mathcal{C}$, either $e_3$ is adjacent
		to both $e_1$ and $e_2$, or not adjacent to any of these edges.
		In both cases, $\varphi$ is extendable by coloring all
		edges in the dimensional matching containing $e_3$ by color $3$.
		If, on the other hand, $e_1$ and $e_2$ do not lie on a common $4$-cycle,
		then  we may extend $\varphi$ by coloring all
		edges of the dimensional matching containing $e_3$ by color $3$.
	This completes the base step of our inductive proof of
	Theorem \ref{th:dprecol}.
	
	\bigskip
	
	Let us now assume that the theorem holds for any hypercube of dimension
	less than $d$, and consider a precoloring $\varphi$ of $Q_d$.
	The induction step of the proof of
	Theorem \ref{th:dprecol} is done by proving a series of lemmas.
	We shall also need two preparatory lemmas.
	
	\begin{lemma}
		\label{cl:onecolor}
			Let $Q_{d-1}$ be the $(d-1)$-dimensional hypercube, 
			where $d-1\geq 3$.
			Suppose that $d-1$ edges are precolored with color $1$
			in $Q_{d-1}$, 
			and that
			there is a vertex $u$ not incident with any precolored edge,
			but every neighbor of $u$ is incident with an edge colored
			$1$. Let $e_1$ be an uncolored edge which is not incident with $u$,
			but adjacent to at least one precolored edge.
			Unless $d-1 =3$ and one end $x$ of $e_1$ is incident with three
			uncolored edges all of which are adjacent to precolored edges,
			then there
			is a cycle $C=v_1v_2 \dots v_{2k} v_1$ in $Q_{d-1}$
			of even length 
			with the following properties:
			
			\begin{itemize}
			
			\item[(i)] $v_1v_2 = e_1$ and $u \notin V(C)$,
			
			\item[(ii)] none of the edges in
			$\{v_1v_2, v_3v_4, \dots, v_{2k-1}v_{2k} \}$ are precolored,
			
			\item[(iii)] if any vertex in $\{v_1,\dots,v_{2k}\}$ is incident
			with a precolored
			edge, then this edge lies on $C$.
			\end{itemize}
			
		\end{lemma}
		
		\begin{proof}
			Let $M_1, \dots, M_{d-1}$ be the $d-1$ dimensional matchings
			in $Q_{d-1}$ and let $e_1 =wx \in M_1$. 
			Let $e_2= vw \in M_2$ be a precolored 
			edge adjacent to $e_1$. 
			
			We first consider the case when $e_1$ is adjacent to
			two precolored edges.
			If the other precolored edge $e_3$
			adjacent to $e_1$ is in $M_2$, then 
			$v$ is adjacent to an endpoint of $e_3$ via an edge
			from $M_1$, so
			there
			is trivially a $4$-cycle satisfying (i)-(iii).
			So we assume that $e_3 \in M_3$.
			Moreover, since $Q_{d-1}$ has no odd cycles, we may
			without loss of generality
			assume that
			$v$ and $x$ are both adjacent to $u$. Since any $4$-cycle
			has edges from exactly two dimensional matchings
			(which e.g. follows from Proposition \ref{prop:HavelMoravek} (ii)),
			this implies that $uv \in M_1$ and $ux \in M_2$.
			
			Consider the subgraph of $Q_{d-1}$ induced by the edges in
			$M_1 \cup M_2 \cup M_3$; by Lemma \ref{lem:DimMathInduce},
			this is a disjoint union of
			$3$-dimensional hypercubes. Let $F$ be the component
			of this subgraph containing $e_1$, $e_2$ and $e_3$.
			Since any precolored edge is adjacent to an edge incident with $u$,
			it follows that the edge of $M_3$ incident with $u$
			is adjacent to some precolored edge $e'$ that lies in $M_1$
			or in $M_j$ for some $j \geq 4$. 
			Moreover, 
			$e_3, e'$ and $e_2$ are the only precolored edges incident
			with vertices of $F$. 
			If $e' \in M_1$, then there is a $6$-cycle in $F$ containing
			$e_1, e_2, e_3, e'$ that satisfies (i)-(iii);
			if $e' \notin M_1$ then there is a $6$-cycle in $F$
			containing $e_1, e_2, e_3$, but no vertex incident with $e'$,
			which satisfies (i)-(iii).

			Suppose now that
			$e_1 =wx$ is adjacent to precisely one precolored edge $e_2 = vw$. 
			Since every precolored edge is adjacent to an edge incident with $u$,
			either $v$ or $w$ is adjacent to $u$.
			Let us first assume that $w$ is adjacent to $u$.
			Since $x$ is not incident to any precolored edge,
			and all precolored edges are adjacent to edges incident with $u$,
			the unique vertex $a \notin \{w,x,v\}$ 
			in the component of the subgraph $Q_d[M_1 \cup M_2]$ containing $e_1$
			is not incident with a precolored edge.
			Thus,
			there is a $4$-cycle $vwxav$ whose edges lie in $M_1 \cup M_2$
			and which satisfies (i)-(iii).		
			
			Let us now consider the case when $v$ is adjacent to $u$.
			Then we may assume that $e_3 = uv$ is in some dimensional matching
			distinct from $M_1$ and $M_2$, since $uv \in M_1$ implies that
			$x$ is adjacent to $u$
			and thus $x$ is incident with some precolored edge, contradicting our
			assumption.
			We assume $e_3 \in M_3$.
			As above we consider the
			subgraph of $Q_{d-1}$ induced by the
			edges in $M_1 \cup M_2 \cup M_3$.
			Let $F$ be the component of this induced subgraph
			containing $e_1$, $e_2$ and $e_3$.
			Straightforward case analysis shows that there is a
			$4$- or $6$-cycle satisfying (i)-(iii) unless
			every edge incident with $x$ in $F$ is 
			adjacent to a precolored edge of $F$.
				It remains to prove that if $d-1 \geq 4$, and every edge
			incident with $x$ in $F$ is adjacent to a precolored edge of $F$,
			then there is a cycle
			$C$ satisfying (i)-(iii).
			Consider the subgraph
			of $Q_{d-1}$ induced by $M_1 \cup M_2 \cup M_3 \cup M_4$.
			Let $K$ be the component of this induced subgraph containing
			$F$. Since 
			all
			precolored edges are adjacent to edges incident with $u$,
			$K$ contains at most one precolored edge not in $F$.
			Using these facts, it is straightforward that $K$ has
			a cycle containing all three precolored
			edges of $F$ and satisfying (i)-(iii).
		\end{proof}
	
	\begin{lemma}
		\label{cl:PartCol}
			Let $\varphi_1$ be an edge precoloring of $d-1$ 
			edges of $Q_{d-1}$ such that there is a vertex $u$ incident
			with an edge $e'$ precolored $2$, and where every other
			edge incident with $u$ is not precolored but adjacent
			to an edge precolored $1$. Let $e_1$ be some
			edge precolored $1$ in $Q_{d-1}$.
			There is a partial proper edge coloring
			$f_1$ of $Q_{d-1}$ with colors $1$ and $2$
			satisfying the following:
			
			\begin{itemize}
			
				\item[(i)] Any vertex of $Q_{d-1}$ is incident 
				with at least one edge that is
				colored under $f_1$;
				
				\item[(ii)] The coloring $f_1$ agrees 
				with $\varphi_1$ on any edge that is colored
				under $\varphi_1$;
				
				\item[(iii)] $e_1$ is contained in a cycle
				that is $(1,2)$-colored under $f_1$, and
				which does not contain $e'$.

			\end{itemize}
			
		\end{lemma}

		\begin{proof}
		Note that the conditions of the claim implies
		that $e_1$ is no incident with $u$, but an end  of $e_1$ is adjacent to $u$. 
		Let $M_1, M_2, M_3$ be three dimensional matchings 
			in $Q_{d-1}$
			that contain
			$e_1$, $e'$ and an edge adjacent to both $e'$ and $e_1$.
		
			The spanning subgraph of $Q_{d-1}$ induced by 
			$M_1 \cup M_2 \cup M_3$ is a disjoint union
			of copies of $Q_3$; let $F$ be the component 
			containing $e_1$ and $e'$.
			
			If $e_1$ and $e'$ lie in distinct dimensional matchings,
			then it is easy to see that there is
			a $4$-cycle $C_1$ in $F$ containing $e_1$ and no other
			precolored edge, and that satisfies that no vertex of $C_1$ is incident to
			a precolored edge that is not in $C_1$.
			We
			color the edges of $C_1$ by colors $1$ and $2$ alternately
			such that the coloring agrees with $\varphi_1$. 
			Additionally we retain the color 
			of any precolored edges of $F$, and we possibly 
			color one additional edge in
			$F$ by color $2$ so that every vertex of $F$ is 
			incident with a colored edge.
			Denote the obtained coloring of $F$ by $h_1$.
			
			Now, 
			since every precolored edge has one endpoint
			adjacent to $u$, every
			component $T$ of $Q_{d-1}[M_1 \cup M_2 \cup M_3]$ 
			distinct from $F$ contains at most one 
			precolored edge.
			Hence,
			there is
			a perfect matching $M_T$ of $T$ that does not 
			contain any precolored edge.
			We extend $h_1$ to a coloring of $Q_{d-1}$ satisfying (i)-(iii) by retaining
			the color of any 
			$\varphi_1$-precolored edge not in $F$, and for every component
			$T$ of $Q_d[M_1 \cup M_2 \cup M_3]$ distinct from $F$ we color
			every edge in $M_T$ by color $2$.

			Suppose now that $e_1$ and $e'$ lie in the same dimensional matching, 
			$M_1$ say. Then $e_1$ and $e'$ is contained in a $4$-cycle
			of $F$. Suppose that the edges of this cycle are in $M_1 \cup M_3$.
			If $M_3 \cap E(F)$ contains no $\varphi$-precolored edge, then $e_1$
			is contained in a $4$-cycle such that no vertex of this cycle
			is incident with another $\varphi$-precolored edge.
			On the other hand, if $M_3 \cap E(F)$ contains some precolored edge,
			then $e_1$ is contained in a $6$-cycle $C_2$ not containing
			$e'$, but two other precolored edges colored $1$.
			Moreover, no vertex of $C_2$ is incident to
			a precolored edge that is not in $C_2$.
			Thus there is a proper edge coloring $h_2$ of $C_2$ 
			with colors $1$ and $2$
			that agrees with $\varphi_1$.

			The coloring $h_2$ can be extended to a partial proper edge coloring
			of $Q_{d-1}$ satisfying (i)-(iii) by 
			proceeding as above.
		\end{proof}
	
	We now turn to the induction step of the proof of Theorem \ref{th:dprecol}.
	Henceforth, we shall always assume that $\varphi$ is a 
	proper $d$-edge precoloring
	of precisely $d$ edges in $Q_d$. Moreover, we assume that $M$ is a
	dimensional matching in $Q_d$ and that $H_1$ and $H_2$ are the components
	of $Q_d-M$; so $H_1$ and $H_2$ are both isomorphic to $Q_{d-1}$.
	As in the proof of Theorem \ref{th:hypercube},
	two edges of
	$H_1$ and $H_2$ are corresponding if their endpoints are joined
	by two edges of $M$. Similarly, two vertices are corresponding
	if they are joined by an edge of $M$.
	
	In the proofs of the lemmas we shall generally distinguish between
	the cases when there is a dimensional matching that contains
	no precolored edge, and when there is no such dimensional matching.

	
	\begin{lemma}
	\label{lem:mono}
		If all $d$ precolored edges in $Q_d$ have the same color
		and $\varphi \notin \mathcal{C}_3$, then $\varphi$
		is extendable.
	\end{lemma}
	\begin{proof} 
		Suppose that the color used by $\varphi$ is $1$.
		It follows from K\"onig's edge coloring theorem
		that for proving the lemma, it suffices to show that
		there is a perfect matching in $Q_d$ containing all edges
		precolored $1$.
		
		\bigskip
		
		\noindent
		{\bf Case 1.} {\em Every dimensional matching contains a precolored edge:}
		
		\medskip
		
		The assumption implies that $M$ 
		contains precisely one edge $u_1 u_2$ colored $1$,
		where $u_i \in V(H_i)$.

		\bigskip
		
		\noindent
		{\bf Case 1.1.} {\em No precolored edges are in $H_2$:}
		
		\medskip
		
		The conditions imply that
		$d-1$ precolored edges are in $H_1$. By coloring
		the edges of $H_2$ corresponding to the
		precolored edges of $H_1$ by color $1$, 
		coloring all edges of $M$ that are not adjacent to any colored
		edges by color $1$, we obtain a partial coloring where the precolored
		edges form a perfect matching of $Q_d$; thus $\varphi$ is extendable.
		
		\bigskip
		
		\noindent
		{\bf Case 1.2.} {\em Both $H_1$ and $H_2$ contain at most $d-3$
		precolored edges:}
		
		\medskip
		
		Suppose that there is a vertex $x_1$ of $H_1$ adjacent to $u_1$
		such that neither $x_1$ nor the vertex $x_2$ of $H_2$ corresponding 
		to $x_1$
		is incident with a precolored edge. Consider 
		the precoloring of $H_1$ obtained
		from the restriction of $\varphi$ to $H_1$
		by in addition coloring $x_1u_1$ with $1$.
		By Theorem \ref{th:hypercube}, this precoloring
		is extendable
		to a proper $(d-1)$-edge coloring $f_1$ of $H_1$; and similarly
		there is an extension $f_2$
		of the precoloring of $H_2$ obtained from the 
		restriction of $\varphi$ to $H_2$ by in addition coloring
		$u_2x_2$ by color $1$; this
		is evident since the obtained precolorings of $H_1$ and $H_2$, respectively,
		both contain at most $d-2$ precolored edges.
		We now define a perfect matching containing
		all $\varphi$-precolored edges of $Q_d$ by removing 
		$u_1x_1$ and $u_2x_2$
		from the union of all edges colored $1$ under $f_1$ or $f_2$,
		and adding the edges $u_1u_2$ and $x_1x_2$. We conclude that
		$\varphi$ is extendable.
		
		Now suppose that for each neighbor $x_1$ of $u_1$ either
		$x_1$ or the corresponding vertex $x_2$ of $H_2$
		is incident with a precolored
		edge. Since $Q_d$ is $d$-regular and contains altogether
		$d$ precolored edges, this implies that all precolored edges have one
		end which is adjacent to
		either $u_1$ or $u_2$.
		Now, since $Q_d$ contains $d$ precolored edges,
		$M$ contains one precolored edge,
		and both $H_1$ and $H_2$ contain at most $d-3$ 
		precolored edges, $(d-3)+(d-3)+1 \geq d$, and so $d \geq 5$. Thus
		$u_1$ is adjacent to at least two 
		vertices incident with precolored
		edges in $H_1$, and $u_2$ is adjacent to two vertices 
		of $H_2$ incident
		with precolored edges. 
		
		We shall need the following claim.
		
		\begin{claim}
			There is a dimensional matching $M_j$
			and a precolored edge $vv' \in M_j$
			such that not every other precolored
			edge has one end adjacent to either $v$ or $v'$.
		\end{claim}
		
		\begin{proof}
		Recall that Proposition \ref{prop:HavelMoravek}
		holds if we take the dimensional matchings of $Q_d$
		as the colors in the proposition.
		Let $M_1,\dots, M_d$
		be the dimensional matchings in $Q_d$, where $M_1=M$. Without
		loss of generality, we assume that there are
		precolored edges $e_j=a_jb_j \in M_j$ and
		$e_k=a_kb_k \in M_k$,
		such that $b_j$ and $u_1$ are adjacent and $u_1b_j \in M_2$,
		and $b_k$ and $u_1$ are adjacent and $u_1b_k \in M_3$.
		If no endpoint of $e_j$ is adjacent to an endpoint of $e_k$, then
		we are done, so suppose, without loss of generality, that
		$a_j$ and $b_k$ are adjacent. By Proposition \ref{prop:HavelMoravek} (ii),
		this means that $a_jb_j \in M_3$ and $a_jb_k \in M_2$.
		Now, $H_2$ contains at least one precolored edge 
		$ab$, where either $a$ or $b$ is adjacent
		to $u_2$ via an edge from a dimensional matching that is distinct
		from $M_2$ and $M_3$, because otherwise, as for $H_1$, it would follow
		that at least one precolored edge of $H_2$ would be in $M_2$ or $M_3$;
		a contradiction to the assumption that all precolored edges are in
		distinct dimensional matchings.
		Thus, without loss of generality, we assume that
		$u_2a \in M_4$. Moreover, since all precolored edges lie 
		in distinct dimensional matchings $ab \notin M_1 \cup M_3$.
		Hence, all edges on the path $a_jb_ju_1u_2a$ are in different
		dimensional matchings.
		Again using Proposition \ref{prop:HavelMoravek} (ii), it thus follows
		that no endpoint of $ab$ is adjacent to an endpoint of $a_jb_j$.
		We conclude that there is a dimensional matching $M_j$
		and a precolored edge $vv' \in M_j$
		such that not every other precolored
		edge has one end adjacent to either $v$ or $v'$.
		\end{proof}
		
		Let $M_j$ be a dimensional matching as in the preceding claim.
		Then the graph $Q_d -M_j$,
		consists of two
		copies $J_1$ and $J_2$ of $Q_{d-1}$.
		Moreover, if
		both $J_1$ and $J_2$ contain at most $d-3$ precolored edges,
		then we may proceed as above for obtaining an extension
		of $\varphi$. Moreover, if $d-1$ precolored
		edges lie in $J_1$, then we proceed as in Case 1.1.
		We conclude that it suffices to consider the case when
		$d-2$ edges of $H_1$ (or $H_2$) are precolored.

			\bigskip
		
		\noindent
		{\bf Case 1.3.} {\em $H_1$ contains $d-2$ precolored edges and $H_2$ 
		contains one	precolored edge:}
		
		\medskip

		Denote by 
		$v_2w_2$ the precolored edge of $H_2$ and let 
		$v_1$ and $w_1$ be the vertices of $H_1$ corresponding
		to $v_2$ and $w_2$, respectively. If no precolored edge
		is incident with $v_1$ or $w_1$, then we may color $v_1w_1$ with color
		$1$, and then color all edges of $H_2$ corresponding to precolored
		edges of $H_1$ by color $1$.
		The resulting coloring is extendable, since by coloring
		any edge of $M$ (including $u_1u_2$), which is 
		not adjacent to a colored edge, by color $1$,
		the precolored edges form a perfect matching of $Q_d$,
		as required.
		
		Thus, we may assume that some $\varphi$-precolored 
		edge in $H_1$ is incident
		with $v_1$ or $w_1$, say $w_1$. 
		Since there are $d-2$ precolored edges in $H_1$,
		the restriction of $\varphi$ to $H_1$ is extendable; in 
		particular, there is a perfect matching $M^*$ in $H_1$ 
		containing all
		precolored edges of $H_1$.
		Note that the edge of $M^*$ incident
		with $u_1$ is not incident with $w_1$.
		If $u_1v_1 \notin M^*$, then let
		$e'$ be the edge of $H_2$ corresponding to the edge of $M^*$
		incident with $u_1$. Then the precoloring of $H_2$ where
		$e'$ and $v_2w_2$ are colored $1$ is extendable, in particular
		there is perfect matching $M^*_2$
		in $H_2$ containing both these edges.
		By removing the edge $e'$ from $M^*_2$, removing the corresponding
		edge from $M^*$ and including two edges from $M$, we obtain a perfect
		matching in $Q_d$ containing all precolored edges of $\varphi$;
		hence, the coloring $\varphi$ is extendable.
		
		Thus, we may assume that $u_1v_1 \in M^*$, and, consequently,
		$v_1$ is not incident to any $\varphi$-precolored edge.
		Moreover, if $u_1$ is the only neighbor of $v_1$ that is
		not incident with a precolored edge of $H_1$, then
		$\varphi \in \mathcal{C}_3$, because all neighbors
		of $v_1$ are incident with a precolored edge in $Q_d$.
		Thus, there is a neighbor $y \neq u_1$
		of $v_1$ in $H_1$ that is not incident with any precolored edge.

		Consider the precoloring $\psi$ of $H_1$ obtained from the restriction
		of $\varphi$ to $H_1$
		by also coloring $v_1y$ by color $1$. If $\psi$ is extendable
		to a proper $(d-1)$-edge coloring $\psi'$ of $H_1$
		then in the matching of $H_1$ containing all edges with
		color $1$ under $\psi'$,
		$u_1$ is matched to some vertex distinct from $v_1$, and, as before,
		this implies that $\varphi$ is extendable.
		Thus it suffices to consider the case when $\psi$ is not extendable
		to a proper edge coloring of $H_1$.
		Since there are exactly $d-1$ precolored edges under $\psi$,
		all of which have the same color,
		by the induction hypothesis, there is some vertex $a$
		of $H_1$ that is not incident with any $\psi$-precolored edge,
		but all neighbors of $a$ are incident with $\psi$-precolored edges.
		We shall prove that this property also holds
		for the vertex $u_1$ unless $\varphi$ is extendable.
		
		\begin{claim}
			Every neighbor of $u_1$ in $H_1$ is incident with
			a $\psi$-precolored edge unless $\varphi$ is extendable.
		\end{claim}
		
		\begin{proof} 
		Assume to the contrary that $u_1$
		does not have this property. Then there is a
		neighbor $z \neq v_1$ of $u_1$ that is not incident
		to any $\varphi$-precolored edge.
		Let $\alpha$
		be the precoloring of $H_1$ obtained from
		the restriction of $\varphi$ to $H_1$ by coloring the edge
		$u_1z$ by color $1$. As we have seen above, if any of the
		precolorings $\psi$ or $\alpha$ of $H_1$
		is extendable (in $H_1$) to a proper $(d-1)$-edge coloring,
		then $\varphi$ is extendable (because in both these extensions
		$u_1$ is matched to some other vertex than $v_1$ in the matching
		induced by color $1$.)
		
		We conclude that since neither of $\alpha$ and $\psi$ are extendable,
		there are vertices $b_1$ and $b_2$ such that under $\alpha$
		every neighbor of $b_1$ in $H_1$ is incident with a precolored
		edge, and under $\psi$ every neighbor of $b_2$ in $H_1$ is incident
		with a precolored edge. Note that $b_1 \neq b_2$ because
		the vertices $u_1,v_1,y,z$ are all distinct and all
		vertices in $H_1$ have degree $d-1$ in $H_1$. 
		Since $d-1 \geq 3$, $b_1$ and $b_2$ are both
		adjacent to endpoints of at least two distinct $\varphi$-precolored
		edges. Hence, the distance $d(b_1,b_2)$ between $b_1$ and $b_2$
		is at least $1$ and at most $3$.
		We consider some different subcases.
		
		\bigskip
		
		\noindent
		{\bf Subcase A.} {\em $d(b_1,b_2)=1$:}
		
		\medskip
		
		Since $d(b_1,b_2)=1$ and $b_1$ and $b_2$ are both adjacent
		to endpoints of at least two distinct $\varphi$-precolored
		edges $e_1$ and $e_2$ in $H_1$, there are two $4$-cycles containing
		$e_1$ and $b_1b_2$, and $e_2$ and $b_1b_2$, respectively.
		However, this implies that $e_1$ and $e_2$ are in the same dimensional
		matching; a contradiction to the assumption of Case 1.
		We conclude that the case $d(b_1,b_2)=1$ is not possible.
		
		\bigskip
		
		\noindent
		{\bf Subcase B.} {\em $d(b_1,b_2)=2$:}
		
		\medskip
		
		In this case, it follows that $b_1$ and $b_2$ have a common neighbor which
		is incident to an  edge which is precolored under $\varphi$. 
		Then, since $H_1$ is bipartite, $b_1$
		and $b_2$ are adjacent to the same end of every edge which is
		precolored under $\varphi$. 
		If $d-1 =3$, then $H_1$ contains two $\varphi$-precolored edges
        that lies in the same dimensional matching, because $b_1$ and $b_2$
        lie on a common  $4$-cycle with edges from exactly two dimensional matchings;
        a contradiction to the assumption of Case 1.
		If $d-1 \geq 4$, then $H_1$ has at least $3$ $\varphi$-precolored edges,
		and thus
		two adjacent edges of $H_1$ lie on at least two
		distinct $4$-cycles; a contradiction because $H_1$
		is isomorphic to $Q_{d-1}$.
		We conclude that the case $d(b_1,b_2)=2$ is not possible.
		
		\bigskip
		
		\noindent
		{\bf Subcase C.} {\em $d(b_1,b_2)=3$:}
		
		\medskip
		
		If $d(b_1, b_2)=3$, then
		$b_1$ and $b_2$ are adjacent to distinct
		ends of an edge which is precolored under $\varphi$.
		Since $H_1$ is bipartite, this implies that $b_1$ and $b_2$
		are adjacent to distinct endpoints of every edge that is
		precolored under $\varphi$. 
		If $d-1=3$, then $H_1$ contains two $\varphi$-precolored edges,
		and there is exactly one edge of $H_1$ that we can color $1$
		so that $b_1$ or $b_2$ is adjacent to three vertices all of which
		are incident with an edge colored $1$. This contradicts that
		the vertices $u_1, v_1, y, z$ are all distinct.
		
		Assume now that $d-1 \geq 4$. Then $b_1$ and $b_2$
		are adjacent to distinct endpoints of at least three
	    $\varphi$-precolored edges that lie in distinct dimensional
	    matchings.
		In fact, we must have
		$d-1 = 4$. Indeed, recall that Proposition 
		\ref{prop:HavelMoravek} holds if we take the colors to be the
		dimensional matchings of $Q_d$.
		It then follows from Proposition \ref{prop:HavelMoravek} (ii)
		that two vertices in a hypercube are endpoints
		of at most three distinct paths of length $3$, where
		any two central edges of the paths are in distinct dimensional
		matchings.
		Furthermore, since all edges of these three distinct paths 
		with endpoints $b_1$ and $b_2$ must lie in three 
		distinct dimensional matchings (which again follows
		from Proposition \ref{prop:HavelMoravek} (ii)),
		these paths induce a hypercube $F$ of dimension $3$.
		Now, since in $H_1$, $u_1$ is adjacent to at least two vertices 
		that are not incident with any $\varphi$-precolored edges,
		$u_1 \notin V(F)$. Moreover, $v_1 \notin \{b_1, b_2\}$,
		because $v_1$ has at least two neighbors that are not
		incident with any $\varphi$-precolored edges of $H_1$.
		Now, since $d-1=4$, and all $\varphi$-precolored edges of $H_1$
		are in $F$, this implies that
		there is a perfect matching of $H_1$ containing
		all $\varphi$-precolored edges of $H_1$, and where
		$u_1$ is matched to some other vertex than $v_1$; 
		as before,
		this implies that $\varphi$ is extendable.
		\end{proof}

		From the preceding claim,
		we conclude that we may assume that $u_1$ is not incident to any
		$\psi$-precolored edge, but every neighbor of $u_1$ is incident
		with a $\psi$-precolored edge.
		
		Now, since all $\varphi$-precolored edges of $H_1$ are also 
		$\psi$-precolored,
		both ends of $v_1w_1$ are incident with $\psi$-precolored edges.
		Hence, by Lemma \ref{cl:onecolor}, there is a
		cycle $C=a_1a_2 \dots a_{2k} a_1$ of even length such that
		
		\begin{itemize}
		
			\item[(i)] $a_1 =v_1, a_2=w_1$ and $u_1 \notin V(C)$;
			
			\item[(ii)] none of the edges in 
			$\{a_1a_2, a_3a_4, \dots, a_{2k-1}a_{2k} \}$ are 
			$\psi$-precolored in $H_1$,
			
			\item[(iii)] if any vertex in $\{a_1,\dots,a_{2k}\}$ is incident 
			with a precolored
			edge, then this edge lies on $C$.
			\end{itemize}
			
			From the precoloring $\psi$ of $H_1$ we define 
			another precoloring $\psi_1$ of $H_1$ by coloring
			all uncolored edges in $\{a_2a_3, a_4a_5,\dots, a_{2k}a_1\}$
			by color $1$ and retaining the color
			of every other edge. Next,
			we define a precoloring $\psi_2$ of $H_2$ 
			by coloring all edges of $H_2$
			corresponding to the edges in
			$\{a_1a_2, a_3a_4, \dots, a_{2k-1}a_{2k}\}$
			by color $1$; furthermore, for any edge of $H_1$ 
			which is $\psi_1$-precolored
			and does not lie on $C$, we color the 
			corresponding edge of $H_2$ by $1$.
		
			Note that a vertex of $H_2$ is incident with a $\psi_2$-precolored
			edge if and only if the corresponding vertex of $H_1$
			is incident with a $\psi_1$-precolored
			edge. Moreover, any edge in $Q_d$ which is
			precolored under $\varphi$ is also
			precolored under $\psi_1$ or $\psi_2$. Hence,
			we obtain an extension of $\varphi$ from $\psi_1$ and $\psi_2$ by
			coloring any edge of $M$ which is not
			incident with a $\psi_1$-precolored edge
			by color $1$.
		
		\bigskip
		
		\noindent
		{\bf Case 2.} {\em There is a
		dimensional matching containing no precolored edge:}
		
		\medskip

		Without loss of generality, we assume that no edge of $M$
		is precolored.
		
		\bigskip
		
		\noindent
		{\bf Case 2.1.} {\em No precolored edges are in $H_2$:}
		
		\medskip
		
		If all precolored edges lie in $H_1$, then the precoloring
		is extendable, since by coloring the edges of 
		$H_2$ corresponding to the precolored edges of $H_1$ by color $1$, 
		and then coloring the edges of $M$
		not adjacent to precolored edges by color $1$, we obtain
		a monochromatic perfect matching of $Q_d$ which contains all
		$\varphi$-precolored edges of $Q_d$. 
		
		\bigskip
		
		\noindent
		{\bf Case 2.2.} {\em Both $H_1$ and $H_2$ contain at most $d-2$
		precolored edges:}
		
		\medskip
		
		If both $H_1$ and $H_2$
		contain at most $d-2$ precolored edges, then by Theorem \ref{th:hypercube}, 
		the restriction of $\varphi$ to $H_i$ 
		is extendable to $(d-1)$-edge coloring of $H_i$, $i=1,2$;
		thus $\varphi$ is extendable.
		
			\bigskip
		
		\noindent
		{\bf Case 2.3.} {\em $H_1$ contains $d-1$ precolored edges and $H_2$ 
		contains one	precolored edge:}
		
		\medskip
		
		As in Case 1.3,
		we may assume that the edge $v_1w_1$ of $H_1$, 
		corresponding to the precolored
		edge $v_2w_2$ of $H_2$,
		is adjacent to at least one precolored edge of $H_1$,
		since otherwise $\varphi$ is extendable.

		Now, by the induction hypothesis,
		the restriction of $\varphi$ to $H_1$
		is extendable (and thus there is an extension of $\varphi$) unless
		there is a vertex $u \in V(H_1)$ not incident to any precolored edge, 
		and satisfying that all neighbors 
		of $u$ in $H_1$
		are incident with precolored edges.
		Furthermore, if $v_1=u$ or $w_1=u$, then
		clearly $\varphi \in \mathcal{C}_3$, so we assume that 
		$u \notin \{v_1,w_1\}$.

		If $d-1 = 3$, and one end of $v_1w_1$ is not incident to any precolored
		edge, but all neighbors of $v_1$ or $w_1$ are incident with
		precolored edges, then $\varphi \in \mathcal{C}_3$.
		Thus, since $\varphi \notin \mathcal{C}_3$, 
		and $v_1w_1$ is adjacent to at least one precolored edge,
		it follows from Lemma \ref{cl:onecolor} that there is a
		cycle $C=v_1v_2 \dots v_{2k} v_1$ of even length such that
		
		\begin{itemize}
		
			\item[(i)] 
			$v_2=w_1$, $u \notin V(C)$;
			
			\item[(ii)] none of the edges in
			$\{v_1v_2, v_3v_4, \dots, v_{2k-1}v_{2k} \}$ are
			$\varphi$-precolored in $H_1$,
			
			\item[(iii)] if any vertex in $\{v_1,\dots,v_{2k}\}$ is incident
			with a precolored
			edge, then this edge lies on $C$.
			\end{itemize}
			
			We may now finish the proof in this case by proceeding exactly as
			in Case 1.3 above,
			using the cycle $C$ to construct a precoloring of $H_2$.	
	\end{proof}


	\begin{lemma}
	\label{lem:twocolors}
		If only two distinct colors appear in the precoloring $\varphi$ of $Q_d$
		and $\varphi \notin \mathcal{C}$,
		then $\varphi$ is extendable.
	\end{lemma}
	\begin{proof}
		Without loss of generality we shall assume that color $1$ and
		$2$ appear on edges under $\varphi$.
		
			\bigskip
		
		\noindent
		{\bf Case 1.} {\em Every dimensional matching contains a precolored edge:}
		
		\medskip
		
		Without loss of generality, we assume that $M$ contains an edge 
		$e_M = u_1u_2$
		precolored $1$ under $\varphi$, where $u_i \in V(H_i)$.
		
		\bigskip
		
		\noindent
		{\bf Case 1.1.} {\em No precolored edges are in $H_2$:}
		
		\medskip

		Suppose that color $1$ does not appear in the restriction $\varphi_1$
		of $\varphi$ to $H_1$. If $\varphi_1$ is extendable to a proper
		edge coloring of $H_1$ using colors $2,\dots,d$, then we obtain an
		extension of $\varphi$ by coloring $H_2$ correspondingly, and then 
		coloring all edges of $M$ by color $1$. So assume
		that there is no such extension of $\varphi_1$. By the induction
		hypothesis, there is a vertex $u$ in $H_1$ that is not incident with
		any precolored edge, but all vertices in $H_1$ adjacent to $u$ are incident
		with an edge precolored $2$. If $u$ is an endpoint of $e_M$, then
		$\varphi \in \mathcal{C}_2$; so we assume that this is not the case.
		Thus, either there is an edge $e'$ incident with $u_1$ colored
		$2$, or we can select $e'$ to be an arbitrary edge of $H_1$
		that is incident with $u_1$ but not adjacent to any edge precolored $2$.
		In both cases, we define a 
		precoloring $\varphi'_1$
		of $H_1$ by coloring $e'$ by color $1$. 
		Then trivially there is a proper edge coloring $f_1$ of $H_1$ using
		colors $1, 3\dots,d$ that agrees with $\varphi'_1$. From $f_1$,
		we define a proper edge coloring $f'_1$ by recoloring all edges
		that are precolored $2$ under $\varphi$ by color $2$
		and also recoloring $e'$ with color $2$. This yields a coloring
		of $H_1$ that agrees with the restriction of $\varphi$ to $H_1$
		and where color $1$ does not appear at an end of $e_M$. 
		Hence, we may color $H_2$ correspondingly, and then color every edge
		of $M$ by the color in $\{1,\dots,d\}$ 
		missing at its endpoints in order to obtain an 
		extension of $\varphi$.
		
		Suppose now that color $1$ does appear on some edge of $H_1$.
		By removing the color from any edge of $H_1$ that is precolored $1$,
		we obtain a precoloring $\varphi_1$ of $H_1$. 
		By Theorem \ref{th:hypercube},
		there is a proper edge coloring of $H_1$ using colors
		$2,\dots,d$ that agrees with $\varphi_1$. Now, by recoloring any edge
		of $H_1$ that is $\varphi$-precolored $1$ by color $1$, thereafter
		coloring $H_2$ correspondingly, and finally coloring all edges
		of $M$ by the unique color missing at its endpoints, we obtain
		an extension of $\varphi$.

				\bigskip
		
		\noindent
		{\bf Case 1.2.} {\em Both $H_1$ and $H_2$ contain at most $d-3$
		precolored edges:}
		
		\medskip

		The conditions imply 
		that $d \geq 5$. If there is an edge $e_1$ 
		in $H_1$ adjacent to $e_M$, and such that neither $e_1$ nor the corresponding
		edge $e_2$ of $H_2$ is colored under $\varphi$, and neither of $e_1$ and
		$e_2$ is adjacent to an edge precolored $1$ 
		under $\varphi$ distinct from $e_M$,
		then we color $e_1$ and $e_2$ by color $1$,
		and consider the precolorings
		of $H_1$ and $H_2$ obtained from the restriction of $\varphi$ 
		to $H_1$ and $H_2$, respectively,
		along with coloring $e_1$ and $e_2$ by color $1$.
		By Theorem \ref{th:hypercube}, 
		these colorings are extendable to proper $(d-1)$-edge colorings 
		$f_1$ and $f_2$ of $H_1$ and $H_2$, respectively.
		Now, by recoloring $e_1$ and $e_2$ by color $d$ and then coloring all
		edges of $M$ by the color missing at its endpoints we obtain the 
		required extension of $\varphi$.
		
		Now suppose that there are no edges $e_1$ and $e_2$
		as described in the preceding
		paragraph. 
		Since $Q_d -M$ contains exactly $d-1$ precolored edges,
		and $H_1$ and $H_2$ are $(d-1)$-regular bipartite graphs, this implies that
		any edge colored $2$ under $\varphi$
		is adjacent to $e_M$, and any edge colored $1$ under $\varphi$ is
		adjacent to an edge $e'$ that is adjacent to $e_M$.
		Thus either one or two edges in $Q_d$ are colored $2$
		under $\varphi$.
		
		Suppose first that there are (at least) two edges precolored $1$ in 
		$H_1$ or $H_2$, say $H_1$.
		Let $e'_1$ and $e'_2$ be two such edges.
		Consider the subgraph $J_1$ of $Q_d$ induced by all dimensional
		matchings containing an edge precolored $1$. 
		Since there are at most two edges colored $2$ under $\varphi$,
		the maximum degree
		of $J_1$ is $d-1$ or $d-2$. Moreover,
		there is a  proper edge coloring of $J_2=Q_d- E(J_1)$ using $\Delta(J_2)$
		colors, and
		which agrees with the restriction of $\varphi$ to 
		$J_2$, because $J_2$ is a collection of disjoint
		$1$- or $2$-dimensional hypercubes, where every component contains
		at most one precolored edge. Thus, $\varphi$ is
		extendable if there is an extension with $\Delta(J_1)$ colors
		of the restriction $\varphi_1$ of $\varphi$
		to $J_1$ (using distinct colors from the extension of the restriction of
		$\varphi$ to $J_2$). Now, by the induction hypothesis, 
		there is an extension of
		$\varphi_1$ if for no component $T$ of $J_1$ the restriction of $\varphi_1$
		to $T$ satisfies the condition (C3) 
		(with $d-1$ or $d-2$ in place of $d$).
		If there is such a component $T$ of $J_1$, 
		then clearly
		all precolored edges of $J_1$ are in $T$ and there is a vertex
		$u$ of $T$ that is not incident with any precolored edge, but any vertex
		adjacent to $u$ in $T$ is incident with a precolored edge. 
		Thus we may assume
		that $e'_1$, $e'_2$ and $e_M$ are in the same component of $J_1$, 
		and one endpoint
		of all these three edges is adjacent to $u$. 
		Now, if $u$ is adjacent to $u_2$,
		then since $T$ is bipartite, 
		this implies that $e_M$ and $u_2u$ lie on $2$ common $4$-cycles, which is
		not possible since $T$ is isomorphic to a hypercube. On the other hand,
		if $u$ is adjacent to $u_1$, then since $T$ is bipartite, 
		by Proposition \ref{prop:HavelMoravek},
		this implies that $e'_1$ and $e'_2$ lie in the same dimensional matching;
		a contradiction in both cases, so $\varphi$ is extendable.
		
		It remains to consider the case when only one edge in $H_1$ and
		one edge in $H_2$
		is precolored $1$ under $\varphi$. Since at most two edges are precolored
		$2$ under $\varphi$, this implies that $d=5$ and, consequently,
		there are exactly two edges colored $2$ in $Q_d$.
		Suppose that $u_1v_1$ and $u_2v_2$ are the edges colored $2$
		under $\varphi$, where $u_iv_i \in E(H_i)$. Let $M_2$
		be the dimensional matching containing $u_1v_1$, and let
		$H'_1$ and $H'_2$ be the components of $Q_d - M_2$. Note that
		$u_1u_2$ and $u_2v_2$ lie in the same component of $Q_d - M_2$,
		 $H'_1$ say. Let $u'_2v'_2$ be the edge of $H'_2$ corresponding to
		$u_2v_2$; then $u'_2v'_2$ is not precolored under $\varphi$,
		because every dimensional matching contains a single precolored edge.
		Consider the precoloring $\varphi_1$
		of $Q_d$ obtained from the restriction
		of $\varphi$ to $H'_1$
		by recoloring $u_2v_2$ by color $3$, and the precoloring $\varphi_2$
		obtained from the restriction of $\varphi$ to $H'_2$ by also coloring
		 $u'_2v'_2$ by color $3$. 
		 Let us verify that neither of $\varphi_1$ and $\varphi_2$ satisfies
		 any of the conditions (C1)-(C3) (with $4$ in place of $d$).
		 Indeed, 
		 $H'_1$ contains at most four precolored edges colored
		 by exactly two distinct colors, and, moreover, two precolored edges
		 are adjacent; $H'_2$ contains at most three precolored edges.
		Thus, it follows from Theorem \ref{th:hypercube} and
		the induction hypothesis that there are proper edge colorings $f_1$
		of $H'_1$ and $f_2$ of $H'_2$ using colors $1,3,4,5$ that agree
		with $\varphi_1$ and $\varphi_2$, respectively. Now, by recoloring
		$u_2v_2$ and $u'_2 v'_2$ by color $2$ and coloring all edges of
		$M_2$ by the unique color missing at its endpoints, we obtain
		an extension of $\varphi$.

		\medskip

		By symmetry, it remains to consider the case when $H_1$ contains
		$d-2$ precolored edges, and $H_2$ contains 
		one precolored edge.

			\bigskip
		
		\noindent
		{\bf Case 1.3.} {\em $H_1$ contains 
		$d-2$ precolored edges and $H_2$ 
		contains one	precolored edge:}
		
		\medskip

		Suppose first that
		for every edge $e_1$ in $H_1$ that is adjacent to $e_M$, 
		either $e_1$ or the corresponding
		edge $e_2$ of $H_2$ is colored $2$ under $\varphi$, 
		or one of $e_1$ and $e_2$
		is adjacent to an edge colored $1$ distinct from $e_M$.
		If there are at least two edges precolored $1$ in $H_1$, then
		we proceed as in the preceding case and 
		consider the subgraphs $J_1$ and $J_2$ defined as above.
		So suppose instead that there is only one edge precolored $1$ in
		$H_1$; then $d=4$ and $H_1$ contains
	    one edge precolored $1$ and one edge precolored $2$.
		If $H_2$ contains an edge precolored $2$, then 
		since all precolored edges lie in distinct dimensional
		matchings and all edges precolored $2$ are adjacent to $e_M$, there
		is a perfect matching $M^*$ in $Q_d$ containing all edges 
		precolored $1$ and no edge precolored $2$.
		Since $H_1$ and $H_2$ both contains only
		one edge precolored $2$, this implies that
		$\varphi$ is extendable.
		If $H_2$ contains an edge precolored $1$, then
		one may prooced similarly; the details are omitted.

		Let us now consider the case when there is an 
		edge $e_1 \in E(H_1)$ adjacent to $e_M$ and satisfying that neither $e_1$
		nor its corresponding edge $e_2$ in $H_2$ is precolored or 
		adjacent to an edge colored $1$
		in $H_1$ and $H_2$, respectively.
		If the precoloring $\varphi_1$ obtained from the 
		restriction of $\varphi$ to $H_1$
		by in addition
		coloring $e_1$ by color $1$ is extendable 
		to a $(d-1)$-edge coloring of $H_1$, then
		there is a similar extension of $H_2$ of the restriction 
		of $\varphi$ to $H_2$ 
		along with coloring $e_2$ by $1$. By recoloring 
		$e_1$ and $e_2$ by color $d$,
		it is easy to see that there is an extension of $\varphi$.
		Thus we assume that $\varphi_1$ is not extendable.

		Suppose first that $e_1$ is the only edge
		colored $1$ in 
		$H_1$ under $\varphi_1$.
		If the $\varphi$-precolored edge of $H_2$ is colored $2$,
		then $H_1$ and $H_2$ only contain $\varphi$-precolored edges
		with color $2$, and by Theorem \ref{th:hypercube}, for $i=1,2$, the
		restriction of $\varphi$ to $H_i$ is extendable to a
		proper edge coloring of $H_i$ using colors $2,\dots,d$;
		thus $\varphi$ is extendable by coloring all edges of $M$ by color $1$.
		Hence, we may assume that $H_2$ contains a $\varphi$-precolored edge
		of color $1$.
			Note that
		this implies that the precolored edge $e'_2$ of $H_2$
		is not adjacent to $e_M$. 
		Moreover, the corresponding edge $e'_1$
		of $H_1$ is not $\varphi$-precolored, since all precolored
		edges lie in different dimensional matchings.
		Now, since the restriction of $\varphi$ to $H_1$
		consists of $d-2$ precolored edges with colors distinct
		from $1$, Theorem \ref{th:hypercube} yields
		that there is an extension of $H_1$ using colors $2,\dots, d$.
		We color $H_2$  correspondingly.
		Since $e_M$ and $e'_2$ are not adjacent, we now obtain
		an extension of $\varphi$ by recoloring $e'_1$
		and $e'_2$ by color $1$, and thereafter coloring
		all edges of $M$ by the color in $\{1,\dots,d\}$
		missing at its endpoints.

		Now assume that there are several
		edges $\varphi_1$-precolored $1$ in $H_1$.
		Since 
		$\varphi_1$ is not extendable, only two colors are 
		used in $\varphi_1$,
		and there are at least two edges 
		in $H_1$ precolored $1$ under $\varphi_1$,
		there is some vertex $v \in V(H_1)$ such that either
		
		\begin{itemize}
		
			\item[(a)] $v$ is not incident with any $\varphi_1$-precolored
			edge, but
			any edge incident to $v$ is adjacent to some edge 
			$\varphi_1$-precolored  $1$; or
			
			\item[(b)] $v$ is incident with an edge $\varphi_1$-precolored
			$2$ and all other edges incident with $v$ 
			are not $\varphi_1$-precolored 
			but adjacent
			to edges precolored $1$.
		
		\end{itemize}

		\bigskip
		
		\noindent
		{\bf Subcase A.} (a) holds: 
		
		\medskip
		
		If (a) holds, then every $\varphi$-precolored
		edge of $H_1$ is colored $1$ and thus the
		single $\varphi$-precolored edge in $H_2$ is
		colored $2$. Moreover, the restriction of
		$\varphi$ to $H_1$ is by Theorem \ref{th:hypercube}
		extendable to a proper $(d-1)$-edge coloring; in particular
		there is a perfect matching $M^*_1$ in $H_1$
		containing all edges precolored $1$. Let $e''_1$ be
		the edge of $M^*_1$ that is incident with $u_1$, and let
		$e''_2$ be the corresponding edge of $H_2$. 
		Then there is a perfect matching $M^*_2$ in $H_2$
		which does not contain the $\varphi$-precolored
		edge of $H_2$ if it is distinct from $e''_2$.
		We now define a perfect matching $M^*$ of $Q_d$ by
		removing $e''_1$ and $e''_2$ from $M^*_1 \cup M^*_2$
		and adding two edges from $M$ with the same endpoints
		as $e''_1$ and $e''_2$. Since $M^*$ is a perfect matching
		containing all edges colored $1$ under $\varphi$ and no
		edges with color $2$ under $\varphi$, and there is only
		one edge $\varphi$-precolored $2$ in $Q_d$, $\varphi$
		is extendable.
		
		\bigskip

		\noindent
		{\bf Subcase B.} (b) holds: 
		
		\medskip
		
		Suppose now that (b) holds.
		Then $u_1 \neq v$, because $u_1$ is incident with an
		edge colored $1$ under $\varphi_1$.
		Suppose first that $u_1$ is not adjacent to $v$.
		Then $u_1$ and $v$  have a common neighbor $x$,
		 because $H_1$ is $(d-1)$-regular and contains exactly
		$d-1$ $\varphi_1$-precolored edges.
		Moreover, since $u_1$ and $v$ 
		are at distance $2$ and $H_1$ is a $(d-1)$-dimensional hypercube,	$u_1$ and $v$
		have precisely two common neighbors. 
		Now, since $H_1$ is a $(d-1)$-regular bipartite graph
		and (b) holds, this means that
		there are $d-3$ edges of $H_1$ incident with $u_1$ that
		are neither $\varphi_1$-precolored nor adjacent
		to a $\varphi$-precolored edge of $H_1$.
		Thus if $d \geq 5$, then
		there is an edge $e'$
		incident with $u_1$ that is not precolored under $\varphi_1$, 
		and not adjacent
		to an edge of $H_1$ precolored $1$ under $\varphi$, and, moreover, 
		the analogous statement
	    holds for the corresponding edge of $H_2$.
		Now, since 
		\begin{itemize}
		
		\item $e_1$ and $e'$ are adjacent,
		
		\item there is exactly one edge
		$\varphi$-precolored $2$ in $H_1$,
		and
		
		\item $H_1$ contains at least two $\varphi$-precolored
		edges of color $1$ which lie in different dimensional
		matchings, 
		\end{itemize}
		it follows that
		the precoloring obtained
		from the restriction of $\varphi$ to $H_1$ by 
		in addition coloring $e'$ by color $1$
		is extendable to a $(d-1)$-edge coloring of $H_1$,
		and, as above,
		we obtain an extension of $\varphi$
		by constructing a coloring of $H_2$ as in the preceding subcase.
		    Suppose now that $d=4$. Then, since (b) holds,
		and all $\varphi$-precolored edges are in different
		dimensional matchings, there is a perfect matching $M^*$
		in $Q_d$ containing all edges $\varphi$-precolored $1$, but no
		edges precolored $2$ under $\varphi$; thus $\varphi$ is extendable,
		because $H_1$ and $H_2$ both contains at most one edge precolored $2$
		under $\varphi$.

		Now assume that $u_1$ is adjacent to $v$. 
				Note that $u_1v$ is not colored $2$, because
		$e_1$ is incident with $u_1$ and colored $1$ under
		$\varphi_1$, and
		$H_1$ is bipartite and $(d-1)$-regular,
		and contains exactly $d-1$ precolored edges 
		under $\varphi_1$.
		Let $\varphi'_1$ be the precoloring of $H_1$
		obtained from the restriction of $\varphi$ to $H_1$
		by coloring $u_1v$ by color $1$.
		Let $v_2$ be the vertex of $H_2$ corresponding
		to $v$. Note that no edge of $H_2$ incident
		with $u_2$ or $v_2$ is precolored $1$,
		because in the former case this contradicts $u_1u_2$
		being $\varphi$-precolored $1$, and in the latter case
		$\varphi \in \mathcal{C}_2$.
		Let $\varphi'_2$ be the precoloring of $H_2$  obtained
		from the restriction of $\varphi$ to $H_2$
		by in addition coloring (possibly recoloring)
		$u_2v_2$ by color $1$. Then $\varphi'_1$ and $\varphi'_2$
		are extendable to proper $(d-1)$-edge colorings; in particular
		for $i=1,2$,
		there is a perfect matching $M^*_i$
		in $H_i$ containing all $\varphi'_i$-precolored edges
		with color $1$.
		By removing
		$u_1v$ and $u_2v_2$ from $M^*_1 \cup M^*_2$ and
		adding two edges from $M$ instead
		we get a perfect matching $M^*$ of $Q_d$
		that contains all $\varphi$-precolored edges with color
		$1$, but no such edges with color $2$. Now, since $H_1$ and $H_2$
		each contains only one edge $\varphi$-precolored $2$,
		there is an extension of $\varphi$.

		\bigskip
		
		\noindent
		{\bf Case 2.} {\em There is a
		dimensional matching containing no precolored edge:}
		
		\medskip
		
			Without loss of generality, we assume that no edge of $M$
		is precolored.

		\bigskip

		\noindent
		{\bf Case 2.1.} {\em No precolored edges are in $H_2$:}
		
		\medskip
		
		Without loss of generality we assume that there
		are more colors precolored $1$ than $2$. Then
		by Theorem \ref{th:hypercube}, the precoloring
		of $H_1$ obtained from the restriction of $\varphi$ to $H_1$ 
		by removing color $1$ 
		from all edges $e$ 
		with $\varphi(e) =1$, is extendable to a proper edge coloring $f$
		of $H_1$ using colors $2,\dots,d$.
		By recoloring all the edges $e$ with $\varphi(e) =1$
		by color $1$
		we obtain, from $f$, a $d$-edge coloring $f'$ of $H_1$. 
		Moreover, by coloring every edge
		of $H_2$ by the color of its corresponding edge in $H_1$ under $f'$,
		and then
		coloring every edge of $M$ with the color in $\{1,\dots,d\}$
		missing at its
		endpoints, we obtain an extension of $\varphi$.
		
	\bigskip
		
		\noindent
		{\bf Case 2.2.} {\em Both $H_1$ and $H_2$ contain at most $d-2$
		precolored edges:}
		
		\medskip

		By Theorem \ref{th:hypercube}, for $i=1,2$, 
		there is a $(d-1)$-edge-coloring $f_i$ of $H_i$
		that is an extension of the restriction of
		$\varphi$ to $H_i$.
		By taking $f_1$ and $f_2$ together and coloring
		every edge of $M$ by color
		$d$, we obtain an extension of $\varphi$.

				\bigskip
		
		\noindent
		{\bf Case 2.3.} {\em $H_1$ contains $d-1$ precolored edges and $H_2$ 
		contains one	precolored edge:}
		
		\medskip

		Let $e_2$ be the precolored edge of $H_2$, and let $e_1$
		be the edge of $H_1$ corresponding to $e_2$.
		If the restriction of $\varphi$ to $H_1$ is extendable to
		a $(d-1)$-edge coloring of $H_1$, 
		then it follows, as above,
		that $\varphi$ is extendable. 
		So suppose that the restriction of $\varphi$ 
		to $H_1$
		is not extendable. Then, 
		 since only two colors appear in the
		precoloring $\varphi$ and $d \geq 4$, we may 
		without loss of generality assume that
		either
		
		\begin{itemize}
		
			\item[(a)] there is a vertex $u$ incident with an edge 
			$e'$ precolored $2$, and every edge in $H_1$ incident with $u$
			and distinct from $e'$ is not precolored but adjacent
			to an edge precolored $1$, or
			
			\item[(b)] there is a vertex $u$ of $H_1$ such that
			no edge incident with $u$ is precolored, but every vertex
			adjacent to $u$ in $H_1$ is incident with an edge precolored $1$.
		
		\end{itemize}

		\bigskip
		
		\noindent
		{\bf Subcase A.} (a) holds: 
		
		\medskip
		
		Suppose that (a) holds, and
		let $e'$ be the edge in $H_1$
		that is precolored $2$. We shall consider two different subcases.
		
		\bigskip
		
		\noindent
		{\bf Subcase A.1.} $\varphi(e_2)=1$: 
		
		\medskip

		If $e_1$ is incident with $u$, then the conditions imply that
		$\varphi \in \mathcal{C}_2$,
		so we assume that
		$e_1$ is not incident with $u$. If $e'$ is not adjacent to $e_1$, then
		we define $\varphi_1$ to be the precoloring obtained from the restriction
		of $\varphi$ to $H_1$ by removing color $2$ from $e'$. 
		By Theorem \ref{th:hypercube}, $\varphi_1$ is 
		extendable to a proper edge coloring $f_1$ of $H_1$
		using colors $1,3,\dots,d$. Let $\varphi_2$ be the precoloring of $H_2$
		obtained from the restriction of $\varphi$ to $H_2$ by additionally 
		coloring the edge
		of $H_2$ corresponding to $e'$ by color $f_1(e')$; 
		by Theorem \ref{th:hypercube},
		this precoloring is extendable
		to a proper edge coloring using colors $1,3,\dots, d$. Now, by recoloring
		$e'$ and the corresponding edge of $H_2$ by color $2$ and thereafter
		coloring every edge of $M$ by the color missing at its endpoints, 
		we obtain an extension of $\varphi$.
		
		Let us now consider the case when $e_1$ is adjacent to $e'$, but not incident
		to $u$. Then $e_1$ is not precolored under $\varphi$.
		If $e_1$ is not adjacent to any edge precolored $1$ in $H_1$, then 
		we proceed as follows:
		let $\varphi_1$ be the precoloring of $H_1$ obtained
		from the restriction of $\varphi$ to $H_1$ by
		removing color $1$ from all edges $\varphi$-precolored
		$1$.
		Then $\varphi_1$ is extendable to a proper
		edge coloring using colors $2,\dots, d$. By coloring
		$H_2$ correspondingly, and thereafter recoloring
		all edges $\varphi$-precolored $1$ in $H_1$
		with color $1$, recoloring $e_1$ by color $1$,
		and recolor $H_2$ correspondingly, we 
		obtain an extension of $\varphi$ by coloring
		every edge of $M$ by the unique color in $\{1,\dots,d\}$
		missing at its endpoints.

		Finally, assume that $e_1$ is adjacent to $e'$, not incident
		to $u$, but adjacent to some edge precolored $1$ in $H_1$. 
		From $\varphi$ we define a new precoloring $\varphi'$
		of $Q_d$ with $d$ precolored edges
		by removing the color $2$ from $e'$ and coloring
		the edge of $M$ incident with $u$ by color $1$.
		Now, unless $\varphi' \in \mathcal{C}_3$, then by
		Lemma \ref{lem:mono}, $\varphi'$ is extendable; in particular
		there is a perfect matching $M^*$ containing all
		edges $\varphi$-precolored $1$ but not the edge 
		$\varphi$-precolored $2$. Since $Q_d$ contains only one
		edge $\varphi$-precolored $2$, this implies that
		$\varphi$ is extendable; hence, it suffices to prove
		that $\varphi' \notin \mathcal{C}_3$.
		Now, if $\varphi' \in \mathcal{C}_3$, then
		there is a vertex $v$ that is not incident
		with any edge $\varphi'$-precolored $1$,
		but all neighbors of $v$ are incident with
		$\varphi'$-precolored edges of color $1$. 
		Since $H_1$ contains $d-2 \geq 2$ edges with color
		$1$ under $\varphi'$, $v \in V(H_1)$. Moreover,
		since $\varphi'(e_2) = 1$ and the end $x$ 
		of $e_1$ that is not an end of $e'$ is incident with
		an edge that is $\varphi'$-precolored, it follows
		that $v$ must be the common end of $e_1$ and $e'$.
		However, since $e'$ is colored $2$ under $\varphi$,
		this implies that $\varphi \in \mathcal{C}_2$,
		a contradiction.

			\bigskip
		
		\noindent
		{\bf Subcase A.2.} $\varphi(e_2)=2$: 
		
		\medskip

		If $e' = e_1$, then we consider
		the precoloring $\varphi_1$ of $H_1$ obtained from 
		$\varphi$ by removing the color
		from $e'$. This coloring is, by Theorem \ref{th:hypercube}, extendable
		to a proper edge coloring $f_1$ of $H_1$ using colors $1,3,\dots,d$.
		Let $f_2$ be the corresponding coloring of $H_1$. An extension
		of $\varphi$ can now be obtained by recoloring $e_1$ and $e_2$ by color $2$,
		and then coloring every edge of $M$ by the color not appearing
		at its endpoints.
		
		If  $e_1$ is not adjacent to $e'$ and 
		not precolored $1$, then we proceed as in the preceding paragraph,
		except that we color both 
		$e'$ and $e_1$, and their corresponding edges in $H_2$, by color $2$
		in the final step.
		
		Suppose now that $e_1$ is adjacent to $e'$.
		Then $e_1$ is not precolored
		under $\varphi$, because $H_1$ contains
		exactly $d-1$ precolored edges and (a) holds.
		Moreover, since $d \geq 4$,
		there is an edge $e_3 \in E(H_1)$ precolored $1$
		that is not adjacent to $e_1$. Define a precoloring $\varphi_1$
		of $H_1$ from $\varphi$ by removing color $1$ from $e_3$ 
		and recoloring all other edges of
		$H_1$ precolored $1$ under $\varphi$ by color $3$. 
		By Theorem \ref{th:hypercube},
		the precoloring $\varphi_1$ is extendable to a proper edge coloring $f_1$
		of $H_1$ using colors $2,3,\dots,d$. Now, define a precoloring
		$\varphi_2$ of $H_2$ from the restriction of 
		$\varphi$ to $H_2$ by for every
		edge $e$ in $H_1$ precolored $1$ under $\varphi$, coloring the
		corresponding edge of $H_2$ by $f_1(e)$. 
		The precoloring $\varphi_2$ does not satisfy 
		any of the conditions (C1)-(C4) (with $d-1$ in place of $d$),
		because it is not monochromatic, and all precolored
		edges have one end which is at distance $1$ from
		the vertex of $H_2$ corresponding to $u$.
		Hence,
		by the induction hypothesis,
		the coloring $\varphi_2$ is extendable to 
		a proper edge coloring $f_2$ of $H_2$
		using colors $2,3,\dots, d$. From 
		$f_1$ and $f_2$ we define an extension
		of $\varphi$ by recoloring any edge of $H_1$ that is $\varphi$-precolored $1$
		by color $1$, recoloring every edge of $H_2$ corresponding to such an edge
		by color $1$, and thereafter coloring every edge of $M$ by the unique
		color not appearing at its endpoints.
		
		\medskip
		
		Finally, let us consider the case when $e_1$ is 
		precolored $1$ under $\varphi$. 
		Let $\varphi_1$ be the restriction
		of $\varphi$ to $H_1$.
		By Lemma \ref{cl:PartCol}, there is a
		partial proper edge coloring $f_1$
		of $H_1$ satisfying the conditions
		(i)-(iii) of Lemma \ref{cl:PartCol}. Let $E'$ be the set of edges colored
		under $f_1$. The graph $H_1 -E'$ has maximum degree $d-2$ so the coloring
		$f_1$ can be extended to a proper $d$-edge coloring $f'_1$ of $H_1$
		by using K\"onig's edge coloring theorem. 
		Let $f'_2$ be the corresponding coloring of $H_2$, 
		except that we interchange
		colors on the $(1,2)$-colored cycle containing $e_2$.
		Note that for every vertex $x$
		of $H_1$, the same colors appear at $x$ under $f'_1$ 
		and at the corresponding
		vertex of $H_2$ under $f'_2$. Moreover, $f'_1$ 
		and $f'_2$ agrees with $\varphi$.
		Hence, $\varphi$ is extendable.

		\bigskip

		\noindent
		{\bf Subcase B.} (b) holds: 
		
		\medskip
		
		Recall that if (b) holds, then
		there is a vertex $u$ of $H_1$ such that
		no edge incident with $u$ is $\varphi$-precolored, but every vertex
		adjacent to $u$ in $H_1$ is incident with an edge precolored $1$ 
		under $\varphi$.
		Recall that $e_2$ is the unique edge of $H_2$ that is precolored,
		and $e_1$ is the corresponding edge of $H_1$.
		Since two colors appear in $\varphi$, $\varphi(e_2)=2$.
		If $e_1$ is not precolored, then let $f_2$ be an extension
		of the restriction of $\varphi$ to $H_2$ using colors $2,\dots,d$;
		such an extension exists by Theorem \ref{th:hypercube}.
		Let $f_1$ be the corresponding edge coloring of $H_1$.
		From $f_1$ and $f_2$ we obtain an extension of $\varphi$
		by recoloring all edges precolored $1$ under $\varphi$ by color $1$,
		recoloring all corresponding edges of $H_2$ by color $1$, and thereafter
		coloring every edge of $M$ by the unique color in $\{1,\dots,d\}$
		not appearing at its endpoints.
		
		Suppose now that $e_1$ is precolored under $\varphi$; 
		then $\varphi(e_1)=1$. 
        Since $H_1$ contains at least three $\varphi$-precolored
        edges, there are at most two vertices $v_1$ and $v_2$
        of $H_1$
        which are 
        at distance $1$ from $d-1$ vertices all of which
        are incident with edges precolored $1$
        (because otherwise two vertices of distance $2$ lie
        in at least two distinct $4$-cycles, which is not
        possible since $H_1$ is a $(d-1)$-dimensional hypercube).
        Now, since $d-1\geq 3$, there is an 
        edge $e'$ in $H_2$ that is adjacent to $e_2$, 
        and satisfies that the corresponding edge
        of $H_1$ is not incident with $v_1$ or $v_2$.
        This implies that the precoloring 
        $\varphi'$
        obtained from $\varphi$ by coloring $e'$ by color $1$
        and removing color $2$ from $e_2$ is not in 
        $\mathcal{C}_3$; so by Lemma \ref{lem:mono},
        $\varphi'$ is extendable to a proper $d$-edge coloring
        $f$. Now, $f(e') = 1$; so $f(e_2) \neq 1$, and
        since $e_2$ is the only edge colored $2$ under $\varphi$,
        we obtain an extension of $\varphi$ by permuting colors
        in $f$.
	\end{proof}


	\begin{lemma}
	\label{lem:severalcolors}
		If at least three and at most $d-1$ colors appear on edges under $\varphi$,
		and $\varphi \notin \mathcal{C}$,
		then $\varphi$ is extendable.
	\end{lemma}

	\begin{proof}
		Without loss of generality we shall assume that color $1,2$ and
		$3$ appear on edges under $\varphi$, and that color
		$d$ does not appear under $\varphi$.

		\bigskip
		
		\noindent
		{\bf Case 1.} {\em Every
		dimensional matching contains a precolored edge:}
		
		\medskip
		
		Without loss of generality, we assume that $M$ contains an edge
		$e_M$ precolored $1$ under $\varphi$, and first consider
		the case when all other precolored edges are in $H_1$.
		
		\bigskip
		
		\noindent
		{\bf Case 1.1.} {\em No precolored edges are in $H_2$:}
		
		\medskip
		
		Suppose first that color $1$ does not appear in $H_1$.
		If the restriction of $\varphi$ to $H_1$
		is extendable to a $(d-1)$-edge coloring of $H_1$, then
		we may choose such an extension with 
		colors $2,\dots,d$, and thus
		$\varphi$ is extendable. If, on the other hand, 
		the restriction of
		$\varphi$ to $H_1$ is not extendable, 
		then, since at most $d-2$ different colors
		appear in $H_1$, $\varphi$ satisfies (C2) or (C3)
		(with $d-1$ in place of $d$). Hence, there is a vertex
		$u$ such that all edges in $H_1$ incident with $u$  are either
		precolored, or non-precolored and
		adjacent to an edge of a fixed color, say $2$.
		Note that this implies that at least two
		edges in $H_2$ are precolored $2$.
		If $u$ is an endpoint of $e_M$, then 
		$\varphi \in \mathcal{C}$;
		otherwise, assuming $d > 4$, there is either some 
		edge $e'$ adjacent to $e_M$
		that is not colored under $\varphi$ 
		and not adjacent to any edge precolored $2$ under $\varphi$,
		or an edge $e'$ adjacent to $e_M$ and colored $2$. 
		By removing the colors from all edges precolored $2$ 
		under $\varphi$ and coloring $e'$ by
		color $1$, we obtain, from the restriction of $\varphi$ to $H_1$, 
		a precoloring 
		that is extendable to a $(d-1)$-edge coloring of $H_1$,
		because at least two edges in $H_1$ are colored $2$ under $\varphi$.
		Let $f_1$ be an extension
		of this precoloring using colors $1,3,\dots,d$.
		Now, by recoloring $e'$ by color $2$, 
		and also recoloring all (other) edges precolored
		$2$ under $\varphi$ with color $2$, we obtain a proper 
		$d$-edge  coloring of $H_1$.
		By coloring $H_2$ correspondingly and then coloring 
		every edge of $M$ with the color
		missing at its endpoints, we obtain an extension of $\varphi$.
		
		It remains to consider the case when $d=4$. However, it is
		easy to see that if $d=4$ (and thus $H_1$ is isomorphic to $Q_3$)
		there cannot be a vertex $u$ as described
		above and such that all precolored edges lie in different dimensional
		matchings.

		Suppose now that color $1$ appears in $H_1$ under $\varphi$. 
		By removing the color
		from all edges precolored 
		$1$ under $\varphi$ from the restriction of $\varphi$ to $H_1$, we obtain
		an edge precoloring of $H_1$ 
		that is extendable to a proper $(d-1)$-edge coloring
		of $H_1$. Let $f_1$ be an extension of this
		precoloring using colors $2,\dots, d$. By recoloring all edges
		precolored $1$ under $\varphi$
		by color $1$, we obtain an extension of $\varphi$ as above.
		
		\bigskip

		\noindent
		{\bf Case 1.2.} {\em Both $H_1$ and $H_2$ contain at most $d-3$
		precolored edges:}
		
		\medskip
			
		If there is an edge $e_1$ 
		in $H_1$ adjacent to $e_M$, and 
		such that neither $e_1$ nor the corresponding
		edge $e_2$ of $H_2$ is colored under $\varphi$, and neither of $e_1$ and
		$e_2$ is adjacent to an edge precolored $1$ under $\varphi$,
		then we 
		consider the precolorings
		of $H_1$ and $H_2$ obtained from the restriction of $\varphi$ to $H_i$
		along with coloring $e_1$ and $e_2$ by color $1$.
		By the induction hypothesis, 
		these colorings are extendable to $(d-1)$-edge colorings 
		$f_1$ and $f_2$, respectively.
		Now, by recoloring $e_1$ and $e_2$ by color $d$ and then coloring every
		edge of $M$ by the color missing at its endpoints we obtain the 
		required extension of $\varphi$.
		
		Now suppose that there are no edges $e_1$ and $e_2$
		as described in the preceding
		paragraph. Then any edge precolored by a color
		distinct from $1$ under $\varphi$
		is adjacent to $e_M$, and any edge colored $1$ under $\varphi$ is
		adjacent to an edge $e'$ that is adjacent to $e_M$.

		Let	$J_1$ be the subgraph of $Q_d$ induced by all dimensional
		matchings containing edges precolored $1$ or $2$, and let 
		$J_2 = Q_d - E(J_1)$. Suppose that $J_1$ has maximum degree
		$q$. 
		Note that no component $T$ of $J_1$ has the property
		that the restriction of $\varphi$ to $T$ satisfies
		condition (C2) (with $q$ in place of $d$), because
		an edge precolored $2$ is adjacent to an edge precolored
		$1$. Moreover, no component $T$ of $J_2$,
		with the restriction of $\varphi$ to $T$,
		satisfies any of the conditions (C1)-(C4) 
		(with $d-q$ in place of $d$),
		because if all precolored edges of $J_2$ are in $T$,
		then they are all incident to the same endpoint
		of $e_M$.
		Thus, by the induction hypothesis, 
		the restriction $\varphi_1$ of
		$\varphi$ to $J_1$ is extendable to a proper $q$-edge coloring,
		and the restriction $\varphi_2$ of $\varphi$ to $J_2$ is extendable to
		a proper edge coloring with $d-q$ colors. 
		Moreover, since $\varphi_1$ and $\varphi_2$
		use distinct sets of colors, we may use distinct colors
		for the extensions of $J_1$ and $J_2$, respectively; 
		thus we conclude that $\varphi$ is extendable.

				\bigskip
		
		\noindent
		{\bf Case 1.3.} {\em $H_1$ contains $d-2$ precolored edges and $H_2$ 
		contains one	precolored edge:}
		
		\medskip

		If for every edge $e_1$ in $H_1$ adjacent to $e_M$, 
		either $e_1$ or the corresponding
		edge $e_2$ of $H_2$ is precolored under $\varphi$,
		or one of $e_1$ and $e_2$
		is adjacent to an edge colored $1$ distinct from $e_M$,
		then we proceed exactly as in the preceding case and construct
		an extension of $\varphi$ by defining subgraphs $J_1$ and $J_2$ as 
		above.

		Thus we may assume that there is an
		edge $e_1 \in E(H_1)$ such that neither $e_1$
		nor its corresponding edge $e_2$ in $H_2$ is precolored or
		adjacent to an edge colored $1$
		in $H_1$ and $H_2$, respectively.
		If the precoloring $\varphi_1$ obtained from the
		restriction of $\varphi$ to $H_1$
		by in addition
		coloring $e_1$ by color $1$ is extendable 
		to a $(d-1)$-edge coloring of $H_1$, then
		we can obtain an extension of $\varphi$ as follows:
		By Theorem \ref{th:hypercube},
		there is a similar extension of $H_2$ of the restriction 
		of $\varphi$ to $H_2$
		along with coloring $e_2$ by $1$. By recoloring
		$e_1$ and $e_2$ by color $d$,
		it is easy to see that there is an extension of $\varphi$.
		Thus we assume that $\varphi_1$ is not extendable.
		
		Let $e_M = u_1u_2$, and
		suppose first that there is only one edge precolored $1$
		under $\varphi_1$. If the $\varphi$-precolored edge of
		$H_2$ is not colored $1$, then by the induction hypothesis,
		the restriction of $\varphi$ to $H_i$ is extendable ($i=1,2$),
		to proper edge colorings using colors $2,\dots,d$; thus,
		$\varphi$ is extendable. Hence, we may assume that
		color $1$ appears in $H_2$ under $\varphi$.
		Note that
		this implies that the precolored edge $e'_2$ of $H_2$
		is not adjacent to $e_M$. 
		Moreover, the corresponding edge $e'_1$
		of $H_1$ is not $\varphi$-precolored, since all precolored
		edges lie in different dimensional matchings.
		Now, since the restriction of $\varphi$ to $H_1$
		consists of $d-2$ precolored edges with colors distinct
		from $1$, Theorem \ref{th:hypercube} yields
		that there is an extension of $H_1$ using colors $2,\dots, d$.
		We color $H_2$  correspondingly.
		Since $e_M$ and $e'_2$ are not adjacent, we now obtain
		an extension of $\varphi$ by recoloring $e'_1$
		and $e'_2$ by color $1$, and thereafter coloring
		all edges of $M$ by the color in $\{1,\dots,d\}$
		missing at its endpoints.

		Suppose now that color $1$ appears on several 
		edges in $H_1$ under $\varphi_1$.
		Note that since at least three 
		colors, and at most $d-1$ colors, are used by $\varphi$, 
		this implies that $d \geq 5$.
		Since 
		$\varphi_1$ is not extendable and thus satisfies
	    one of the conditions (C1)-(C4),
		and color $1$ appears on several edges under $\varphi_1$,
		there is some vertex $v \in V(H_1)$ such that
		every edge incident with $v$ is $\varphi_1$-precolored
		or adjacent to an edge precolored with $1$.
		
		Since $u_1$ is incident with an edge precolored $1$ under
		$\varphi_1$, $u_1 \neq v$.
		If $u_1$ is not adjacent to $v$, then
		since $d \geq 5$, and any two adjacent edges 
		in $H_1$ are contained in exactly one $4$-cycle,
		there is some edge $e_1'$ in $H_1$ adjacent to $e_M$ 
		such that neither $e'_1$
		nor its corresponding edge $e'_2$ in $H_2$ is precolored 
		or adjacent to
		any edge precolored $1$ under $\varphi$. 
		Let us prove that
		the precoloring $\varphi''_1$ obtained
		from the restriction of $\varphi$ to $H_1$ by 
		in addition coloring $e'_1$ by color $1$
		is extendable to a $(d-1)$-edge coloring of $H_1$.
		Indeed, if $H_1$ contains only one edge that is
		$\varphi$-precolored $1$, then $v$ is incident
		with edges of at least two distinct colors, so
		$\varphi''_1$ does not satisfy (C2) (with $d$ in place of
		$d-1$); if $H_1$ contains at least two edges that are
		$\varphi$-precolored $1$, then
		since $u_1$ and $v$ have exactly two common neighbors
		and any two $\varphi$-precolored edges lie in
		distinct dimensional matchings,
		$\varphi''_1$ does not satisfy (C2).
		Furthermore, the precoloring of $H_2$ obtained from
		the restriction of $\varphi$ to $H_2$ by also
		coloring $e'_2$ by color $1$ is extendable.
		By recoloring
		$e'_1$ and $e'_2$ by color $d$ we obtain an 
		extension of $\varphi$ as before.
		
		If, on the other hand, $u_1$ is adjacent to $v$, then we
		may color $u_1v$ and proceed as above unless the edge 
		$e'_2$
		of $H_2$ corresponding to $u_1v$ is precolored or adjacent
		to an edge precolored $1$. 
		If the latter holds, then $\varphi \in \mathcal{C}$.
		On the other hand, if $e'_2$ is $\varphi$-precolored,
		then let $M'$ be a dimensional matching in $Q_d$ containing a 
		$\varphi$-precolored edge incident with $v$ and colored
		by a color $c$ that only occurs once under $\varphi$;
		such an edge exist since at least three colors
		are used in $\varphi$. 
		Then both components of $Q_d - M'$ satisfies
		that the restriction of $\varphi$ to
		this component is not in $\mathcal{C}$; 
		thus, by the induction hypothesis,
		the restriction of $\varphi$ to $Q_d -M'$
		is extendable to a proper edge coloring of $Q_d -M'$ using colors
		in $\{1,\dots,d\} \setminus \{c\}$. 
		We conclude that $\varphi$ is extendable.

		\bigskip
		
		\noindent
		{\bf Case 2.} {\em There is a
		dimensional matching containing no precolored edge:}
		
		\medskip
		
		Without loss of generality we assume that no edge of $M$
		is precolored.
		
		\medskip

		The case when all precolored edges are in $H_1$, and the case when
		$H_1$ and $H_2$ both contain at most $d-2$ precolored edges
		can be dealt with exactly as in Case 2 of the proof of 
		Lemma \ref{lem:twocolors}.
		Hence, we assume that $H_1$ contains exactly $d-1$ precolored edges.
		We shall assume
		that $e_2$ is the precolored edge
		of $H_2$, $e_1$ is the edge of $H_1$ corresponding to $e_2$,
		and that there is no edge colored $d$ under
		$\varphi$.
		
		If the restriction of $\varphi$ to $H_1$ is extendable to a
		$(d-1)$-edge coloring of $H_1$, then since the same holds
		for the restriction of $\varphi$ to $H_2$, $\varphi$ is extendable to
		a $d$-edge coloring of $Q_d$; so
		assume that the restriction of $\varphi$ to $H_1$ is not
		extendable. Since at least three distinct colors appear under $\varphi$,
		this implies that 
		
		\begin{itemize}
		    
		    \item[(a)] $d=4$, and there is a dimensional matching
		    in $H_1$ with three edges with three different colors; or
		    
		    \item[(b)] there is an edge $uv$ of $H_1$ that is not precolored,
		    but $uv$ is adjacent to an edge colored $i$, 
		    for $i=1,\dots, d-1$; or
		    
		    \item[(c)] 
		there is a vertex $u$ incident to $k$ precolored edges
		and every edge incident with $u$ in $H_1$, which is not precolored,
		is adjacent to an edge precolored by some fixed color $c_1$.
		    
		\end{itemize}

		\bigskip
				
		\noindent
		{\bf Subcase A.} (a) holds: 
		
		\medskip
		
		Without loss of generality
		we assume that $\varphi(e_2)=1$. If $e_1$ is adjacent both to the
	    edge precolored $2$ and to the edge precolored $3$, then it is
	    straightforward that $\varphi$ is extendable (because all precolored
	    edges of $H_1$ lie in the same dimensional matching).
		Otherwise,
		either the edge colored $2$ or the edge colored $3$ is 
		not adjacent to $e_1$, suppose e.g. that this holds for the edge $e'_1$ colored
		$2$. The precoloring obtained from the restriction of $\varphi$ to $H_1$
		by removing color $2$ from $e'_1$ is extendable to a proper edge
		edge coloring $f_1$ using colors $1,3,4$, and the precoloring obtained from the
		restriction of $\varphi$ to $H_2$ by in addition coloring the edge $e'_2$,
		corresponding to $e'_1$, by the color $f_1(e'_1)$ is extendable
		to a proper edge coloring $f_2$ using colors $1,3,4$. Now, by recoloring
		$e'_1$ and $e'_2$ by color $2$, and thereafter coloring all edges of $M$
		by the color missing at its endpoints, we obtain an extension of $\varphi$.

\bigskip
				
		\noindent
		{\bf Subcase B.} (b) holds: 
		
		\medskip

		Without loss of generality,
		we assume that $\varphi(e_2)=1$. If $e_1$ is not precolored and
		not adjacent to the edge $e'_1$ in $H_1$ precolored $1$, then we construct
		an extension of $\varphi$ in the following way: 
		remove color $1$ from all edges
		colored $1$ under $\varphi$. The resulting precoloring of $H_1$
		is, by Theorem \ref{th:hypercube},
		extendable to a proper edge coloring using colors $2,\dots, d$.
		By coloring $H_2$ correspondingly, then recoloring $e_2$, $e'_1$
		and their corresponding edges in $H_1$ and $H_2$, respectively, by color $1$,
		and thereafter coloring every edge of $M$ by the color missing at its endpoints,
		we obtain an extension of $\varphi$.
		
		Suppose now that $e_1$ is precolored or adjacent to $e'_1$.
		Let us first assume that
		there is some precolored edge $e''_1$ of $H_1$ that is not adjacent to
		$e_1$ and not colored $1$.
		Suppose, for instance, that $\varphi(e''_1)=2$. By removing
		the color $2$ from $e''_1$, we obtain a precoloring of $H_1$ that is
		extendable to a proper edge coloring $f_1$ using colors $1,3,4,\dots, d$.
		Moreover, the precoloring of $H_2$ obtained from the restriction of $\varphi$ 
		to $H_2$ by additionally coloring the edge $e''_2$, corresponding to $e''_1$,
		by the color $f_1(e''_1)$
		is extendable to a proper edge coloring $f_2$ using colors
		$1,3,4,\dots,d$. By recoloring $e''_1$ and $e''_2$ by color $2$, and
		thereafter coloring every edge of $M$ by the color missing at its endpoints,
		we obtain an extension of $\varphi$.
		
		Suppose now that all $\varphi$-precolored edges with colors distinct
		from $1$ are adjacent to $e_1$. If $e_1$ is precolored, then
		$\varphi$ satisfies (C2) (with $d-1$ in place of $d$),
		so we may assume that $e_1$ is not precolored; then $e_1=uv$.
		Moreover if $u$ or $v$ is incident with only one precolored edge
		that is colored $1$, then $\varphi$ satisfies (C2), so
		we assume that either $u$ or $v$ is incident with
		two edges precolored $1$ and $2$, respectively, under $\varphi$.
		Now, by removing color $2$ from the edge $e'$ $\varphi$-colored $2$,
		we obtain a precoloring that is extendable to a proper edge coloring
		$f_1$ using colors $1,3,\dots, d$. Moreover, the precoloring obtained
		from the restriction of $\varphi$ to $H_2$ by in addition 
		coloring the edge corresponding to $e'$ by the color $f_1(e')$
		is extendable to a proper edge coloring $f_2$
		using colors $1,3,\dots, d$. From $f_1$ and $f_2$ we obtain an
		extension of $\varphi$ by recoloring $e'$ and its corresponding copy in $H_2$
		by color $2$, and thereafter coloring every edge of $M$ by 
		the color missing at its endpoints.

		\bigskip

		\noindent
		{\bf Subcase C.} (c) holds: 
		
		\medskip
		
		Let us 
		first assume that at least three colors
		appear in the restriction of $\varphi$ to $H_1$. 
		If $e_1$ is not incident with $u$, then
		there is an edge $e' \neq e_1$ in $H_1$, such that
		$\varphi(e')=c$, $e'$ is not adjacent to $e_1$ and 
		$e'$ is the only edge in $H_1$ with color $c$
		under $\varphi$. Suppose that $\varphi(e_2) \neq c$.
		Then by removing the color $c$ from the restriction of $\varphi$
		to $H_1$, we obtain a precoloring $\varphi_1$ that is extendable
		to a proper edge coloring $f_1$ of $H_1$ using colors 
		$\{1,\dots,d\} \setminus \{c\}$. (Note that $f_1(e') = c_1$.)
		Moreover, there is a similar
		extension $f_2$ of the restriction of $\varphi$ to $H_2$ where 
		the edge of $H_2$ corresponding to $e'$ 
		is colored $f_1(e')$. Now, by recoloring 
		$e'$ and its corresponding copy in $H_2$ by color $c$,
		we obtain an extension of $\varphi$ as before.
		
		If $\varphi(e_2) = c$ and $e_1$ is not precolored under $\varphi$, 
		then we proceed similarly as in the preceding
		paragraph,
		except that after constructing the coloring $f_1$
		of $H_1$ as in the preceding paragraph, we define $f_2$ as the coloring
		of $H_2$ corresponding to $f_1$ and then color 
		$e'$ and $e_1$, and the corresponding edges of $H_2$, by color $c$.
		On the other hand, if $\varphi(e_2) = c$ and 
		$e_1$ is precolored under $\varphi$, then
		$\varphi(e_1)=c_1$ because $e_1$ is not incident with $u$; 
		now, since at least three distinct colors
		are used by $\varphi$ on edges in $H_1$, we may
		clearly choose another $\varphi$-precolored
		edge incident with $u$ as our edge $e'$,
		and then proceed as in the preceding paragraph.
		
		Now assume that $e_1$ is incident with $u$. 
		If $\varphi(e_2) =c_1$, then $\varphi \in \mathcal{C}$,
		so we assume that $\varphi(e_2) \neq c_1$.
		If there is a color $c \neq \varphi(e_2)$ 
		appearing on precisely one edge $e' \neq e_1$ of $H_1$,
		then we consider the restriction of $\varphi$ to $H_1$
		where color $c$ is removed, and proceed as before;
		otherwise, since at least three colors appear in $H_1$ under
		$\varphi$, it follows that $e_1$ is not adjacent to 
		any edge precolored $c_1$ under $\varphi$. Thus
		by removing color $c_1$ from any edge in $H_1$
		precolored by color $c_1$ under $\varphi$,
		we obtain a precoloring that is extendable
		to a proper edge coloring of $H_1$ using colors
		$\{1,\dots,d\} \setminus \{c_1\}$. Moreover, there
		is a similar extension $f_2$ of the restriction of $\varphi$
		to $H_2$ where for any edge $e'_2$ corresponding
		to an edge $e'_1$ of $H_1$ with $\varphi(e'_1) = c_1$,
		we have $f_2(e'_2) = f_1(e'_1)$. From $f_1$ and 
		$f_2$ we may construct an extension of $\varphi$ by recoloring
		any such pair of edges by color $c_1$.

		Let us now consider the case
		when only two colors appear in the restriction of
		$\varphi$ to $H_1$. Since at least 
		three colors appear on edges under $\varphi$,
		it follows that $\varphi(e_2)$ does not appear in $H_1$ under $\varphi$.
		Without loss of generality we assume that 
		$\varphi(e_2) = 2$, 
		color $3$ appears
		on exactly one edge $e'$ in $H_1$, and color $1$
		is the third color used by $\varphi$. If $e' \neq e_1$,
		then we consider the precoloring of $H_1$ obtained from 
		the restriction of $\varphi$ to $H_1$ by removing color $3$.
		There is an extension of this precoloring of $H_1$ using
		colors $\{1,\dots,d\} \setminus \{3\}$ such that
		$\varphi(e')=1$. 
		Let $e_2'$ be the edge of $H_2$ corresponding to $e'$. 
		Then the precoloring obtained from the restriction of $\varphi$ to $H_2$
		by additionally coloring $e'_2$ by color $1$ is extendable to a coloring
		using colors $\{1,\dots,d\} \setminus \{3\}$.
		Now, by recoloring $e'$ and $e'_2$ by color $3$, we can construct
		an extension of $\varphi$.
		
		If, on the other hand, $e'=e_1$, then $e_1$ is not adjacent to any
		edge colored $1$. Let $E'$ be the set of edges colored
		$1$ under $\varphi$.
		If $d >4$, then $|E'| \geq 3$ and we recolor all
		edges in $E'$ by colors $2,3,4$ so that at least one
		edge is colored $i$, $i=2,3,4$.
		This yields a precoloring $\varphi_1$ that, by the induction hypothesis,
		is extendable
		to a proper edge coloring $f_1$ of $H_1$ using colors $2,\dots,d$,
		because the precolored edges form a matching which is colored by 
		at least three distinct colors.
		Next, consider the precoloring $\varphi_2$ of $H_2$
		obtained from the restriction of $\varphi$
		to $H_2$ by setting $\varphi_2(e'_2) = f_1(e'_1)$ 
		for any edge $e'_2 \in E(H_2)$ corresponding to an edge $e'_1 \in E'$.
		The $\varphi_2$-precolored edges form a matching consisting
		of $d-1$ edges, where edges corresponding to $E'$ are colored by at least
		three distinct colors, so
		by the induction hypothesis,
		there is an extension $f_2$ of $\varphi_2$,
		where $f_2(e'_2) = f_1(e'_1)$ for any edge $e'_2 \in E(H_2)$ 
		corresponding to an edge $e'_1 \in E'$. We may now obtain an extension of
		$\varphi$ as before.
		
		It remains to consider the case $d=4$.
		By symmetry of the hypercube,
		it suffices to consider the two cases when all edges in $H_1$
		are in the same dimensional matching, and the case when the two edges precolored
		$1$ are in different dimensional matchings, one of which is necessarily the
		same as the dimensional matching containing $e'$. In both cases it is a
		straightforward exercise to check that there is an extension of $\varphi$
		where two edges in $H_1$, and their 
		corresponding edges in $H_2$, are  the only edges colored $4$; 
		and, moreover, these
		two edges of $H_1$ ($H_2$) lie in a dimensional 
		matching with no precolored edges.
	\end{proof}


	\begin{lemma}
	\label{lem:allcolors}
		If all colors $1,\dots,d$ appear under $\varphi$,
		and $\varphi \notin \mathcal{C}$,
		then $\varphi$ is extendable.
	\end{lemma}
	
	\begin{proof}
		Since all colors are present under $\varphi$, every color
		appears on precisely one edge.
		Let us first note that if every dimensional matching contains
		at most one precolored edge, then trivially $\varphi$ is extendable.
		Thus, for the rest of the proof we assume that there is a
		dimensional matching $M$ that does not contain any precolored
		edge. Let $H_1$ and $H_2$ be the components of $Q_d -M$.

		\bigskip

		\noindent
		{\bf Case 1.} {\em No precolored edges are in $H_2$:}
		
		\medskip
		
		If there is some edge $e$ that is not precolored, and
		adjacent to all
		precolored edges in $H_1$, then $\varphi \in \mathcal{C}$.
		On the other hand, if there is a precolored edge $e$
		such that removing the color from $e$ yields a precoloring
		$\varphi_1$ of $H_1$ that is not in $\mathcal{C}_1$ 
		(with $d-1$ in place of $d$) or $\mathcal{C}_4$,
		then the induction
		hypothesis yields that there is an extension $f$ of 
		$\varphi_1$
		using all colors except the removed one. Suppose e.g. that
		the color from $e$ under $\varphi$ was removed in $\varphi_1$;
		then by recoloring $e$ with $\varphi(e)$ and retaining the color
		of every other edge in $H_1$ under $f$, 
		we obtain a proper $d$-edge coloring
		of $H_1$ that is an extension of $\varphi$; by coloring
		$H_2$ correspondingly and then coloring every edge of $M$
		by the color missing at its endpoints, we obtain an extension of $\varphi$.
		
		Now, suppose that $e$ is a precolored edge of $H_1$, 
		and removing
		the color of $e$ yields a coloring $\varphi_1$ that satisfies (C1).
		Let $e'$ be another $\varphi$-precolored edge of
		$H_1$ that is adjacent to a minimum number of other
		$\varphi$-precolored edges of $H_1$. 
		Then the precoloring obtained from $\varphi$
		by removing the color from $e'$ does not satisfy (C4);
		suppose that it
		satisfies (C1). Then
		either $\varphi \in \mathcal{C}_1$, or there are
		non-precolored
		edges $uv, ux \in E(H_1)$ satisfying that $e'$ is incident
		with $v$, $e$ is incident with $x$, and all other precolored edges
		are incident with $u$. Now, 
		since $u$ is incident with at least two precolored edges
		(from different dimensional matchings),
		by instead removing the color
		on a precolored edge incident with $u$, we obtain
		a precoloring that does not satisfy (C1) or (C4).
		We conclude that if $\varphi_1 \in \mathcal{C}_1$,
		then either $\varphi$ is extendable or $\varphi$ satisfies (C1).

		It remains to consider the case when $\varphi_1$
		satisfies (C4).
		Suppose, consequently, that $d=4$ and that removing
		the color from any precolored edge of $H_1$ yields a precoloring
		that satisfies (C4); then the precolored edges of $H_1$
		lie in a dimensional matching $M'$.
		It is easily seen that since all precolored
		edges lie in $M'$, there is a proper $4$-edge coloring of $H_1$
		which agrees with $\varphi$. By coloring $H_2$
		correspondingly and thereafter coloring all edges of
		$M$ by the color in $\{1,2,3,4\}$ missing at its
		endpoints, we obtain an extension of $\varphi$.

		\bigskip
		
		\noindent
		{\bf Case 2.} {\em Both $H_1$ and $H_2$ contain at most $d-3$
		precolored edges:}
		
		\medskip
			
		Note that neither the restriction of $\varphi$ to $H_1$
		nor to $H_2$ satisfies any of the
		conditions (C1)-(C4) (with $d-1$ in place of $d$).
		Moreover, since $H_1$ and $H_2$ contain altogether $d$ precolored edges, 
		$d \geq 6$, and thus both $H_1$ and $H_2$
		contains at least three precolored edges. We consider
		two different cases.
	
		\bigskip

		\noindent
		{\bf Case 2.1.} {\em There is some edge $e$ in $H_1$ (or $H_2$)
		that is precolored and the corresponding edge of $H_2$ ($H_1$)
		is not precolored:}
		
		\medskip
	
		Without loss of generality we
		assume that $e_1$ is such an edge in $H_1$, $\varphi(e_1) = 1$,
		and that $e_2$ is the edge of $H_2$ corresponding to $e_1$.
		Since both $H_1$ and $H_2$ contain precolored edges,
		there is some color which appears in $H_2$ but not in $H_1$.
		Suppose first that some precolored edge of $H_2$ is not
		adjacent to $e_2$. Assume without loss of generality
		that such an edge is precolored $d$ in $H_2$. 
		Then we construct a new precoloring $\varphi'$
		from $\varphi$ by coloring $e_2$ by color $d$, and
		recoloring $e_1$ by color $d$. The restrictions of $\varphi'$
		to both $H_1$ and $H_2$ are, by the induction 
		hypothesis, extendable to proper edge colorings
		using colors $2,\dots, d$, respectively. Now by recoloring
		$e_1$ and $e_2$ by color $1$ we obtain proper edge colorings
		$f_1$ and $f_2$ of $H_1$ and $H_2$, respectively, satisfying
		that the color in $\{1,\dots,d\}$ not appearing at a vertex $v$ of $H_1$
		is also missing at the corresponding vertex of $H_2$. 
		Since for $i=1,2$, $f_i$ is an extension 
		of the restriction of $\varphi$ to $H_i$, $\varphi$ is extendable.
		
		Suppose now instead that every precolored edge of $H_2$
		is adjacent to $e_2$. In fact, we may assume that if
		$e \in E(H_1)$, $e$ is precolored under $\varphi$ and the
		corresponding edge $e'$ of $H_2$ is not precolored under $\varphi$,
		then $e'$ is adjacent to all precolored edges of $H_2$;
		otherwise we proceed as in the preceding paragraph.
		If all precolored edges of $H_2$ are incident with
		a common vertex, then since there are at least three precolored
		edges in $H_i$, $i=1,2$, 
		this means that all precolored edges of $H_1$ are 
		incident with the
		corresponding vertex of $H_1$; and so,
		$\varphi \in \mathcal{C}_1$.
		Assume now that $e_i = u_iv_i$, $i=1,2$, 
		and that both $u_2$ and $v_2$
		are adjacent to precolored edges; since $H_2$ contains
		at least three precolored edges, $e_2$ is the unique edge
		with this property.
		Moreover, since any precolored edge of $H_1$
		satisfies that if the corresponding edge of $H_2$ is not precolored,
		then it is adjacent to all precolored edges of $H_2$,
		it follows that any precolored edge in $H_1$ is incident
		with $u_1$ or $v_1$. This means that the dimensional matching $M_1$
		in $Q_d$ containing $e_1$, contains no other precolored edge.
		Hence, since both $u_2$ and $v_2$ are incident
		with precolored edges, both components of $Q_d-M_1$
		contain at most $d-2$ precolored edges using
		colors $2,\dots, d$.
		Thus,
		by Theorem \ref{th:hypercube}, the restriction
		of $\varphi$ to $Q_d-M_1$
		is extendable to a proper edge coloring of $Q_d-M_1$ 
		using colors $2,\dots,d$.
		By coloring all edges of $M_1$ by color $1$,
		we obtain an extension of $\varphi$.

		\bigskip

		\noindent
		{\bf Case 2.2.} {\em Each precolored edge of
		$H_1$ corresponds to a precolored edge of $H_2$ and
		vice versa:}
	
		\medskip
		
		The conditions imply that
		$H_i$ contains exactly $d/2$ precolored edges.
		
		Suppose first that $d=6$, and let $u_1u_2$ be a precolored edge of $H_1$,
		and $v_1v_2$ be the corresponding edge of $H_2$. Now, since $H_1$
		contain two additional precolored edges which both correspond to
		precolored edges of $H_2$, and $u_1u_2$ is in four $4$-cycles in $H_1$,
		there is a $4$-cycle $u_1u_2u_3u_4u_1$ in $H_1$ such that $u_3u_4$
		is not precolored and the dimensional matching $M_2$, containing
		$u_2u_3$ and $u_4u_1$, does not contain any precolored edge.
		Let $H'_1$ and $H'_2$ be the components of $Q_d-M_2$. 
		Now, since all precolored edges lie on $4$-cycles
		whose non-precolored edges are in $M$, either both or none
		of the precolored edges of such a cycle are in $H'_i$.
		Hence, $H'_i$ contains an even number of precolored edges,
		and so, we may
		proceed as in Case 1 or Case 3
		of the proof of the lemma.

		Now assume that $d \geq 8$.
		If all precolored edges in $H_1$ are incident with
		one common vertex, then $\varphi \in \mathcal{C}$,
		so we assume that this is not the case;
		thus, there are two precolored edges in $H_1$
		(and thus $H_2$) that are not adjacent. In $H_2$ we assume that
		these edges are colored $d/2+1$ and $d$ respectively. Let $v_1 v_2$
		be the edge precolored $d$ in $H_2$ and let $u_1 u_2$ be the
		corresponding edge of $H_1$.
		Without loss of generality, we assume that $\varphi(u_1u_2) = d/2$.
		Now, since there are exactly $d/2$ precolored edges
		in both $H_1$ and $H_2$, $d \geq 8$, and
		each edge in $H_i$ is in $d-2$ $4$-cycles in $H_i$,
		there are $4$-cycles $u_1u_2u_3u_4 u_1$ and $v_1v_2v_3v_4 v_1$
		in $H_1$ and $H_2$, respectively,
		such that
		\begin{itemize}
			\item $u_1u_2$ and $v_1v_2$ are the only precolored 
			edges of these $4$-cycles;
			\item $v_3v_4$ is not adjacent to an edge precolored $d/2+1$.
		\end{itemize}
			We construct a precoloring $\varphi_1$ of $H_1$ from the 
			restriction of $\varphi$ to
			$H_1$ by in addition coloring $u_1 u_4$ and $u_2 u_3$ by color 
			$d/2+1$ and by coloring $u_3u_4$
			by color $d/2$. Similarly, we define a 
			precoloring $\varphi_2$ of $H_2$ from the
			restriction of $\varphi$ to $H_2$ by recoloring $v_1v_2$ 
			by $d/2 +1$, and in addition coloring
			$v_3v_4$ by $d/2 +1$,
			and $v_2v_3$ and $v_1v_4$ by color $d/2$. 
			Note that the obtained precolorings are proper.
			Now, since $d/2+3 \leq d-1$ (because $d \geq 8$) 
			and none of $\varphi_1$ and
			$\varphi_2$ satisfy any of the conditions
			(C1)-(C4) (with $d-1$ in place of $d$), it follows
			from Theorem \ref{th:hypercube} and the induction hypothesis
			that for $i=1,2$,
			there is a proper edge coloring $f_i$ of $H_i$ 
			using colors $1,\dots,d-1$
			that is an extension of $\varphi_i$.
			Now by recoloring all the edges
			$u_2u_3, u_1u_4, v_1v_2, v_3v_4$ by color $d$
			we obtain two proper edge colorings
			such that by coloring every edge of $M$
			by the color in $\{1,\dots,d\}$
			missing at its endpoints, we obtain an extension of $\varphi$.
		
				\bigskip
		
		\noindent
		{\bf Case 3.} {\em $H_1$ contains $d-2$ precolored edges and
		$H_2$ contains $2$
		precolored edges:}
		
		\medskip
		
		We consider two different subcases.
		
			\bigskip

		\noindent
		{\bf Case 3.1.} {\em 
		No precolored edge of $H_1$ satisfies that the corresponding
			edge of $H_2$ is non-precolored:}
		
		\medskip
			
			The conditions imply that
			$d=4$. Without loss
			of generality we assume that $H_1$ contains two edges
			$e_1$ and $e'_1$ precolored $1$ and $2$, respectively.
			Let $e_2$ and $e'_2$ be the corresponding edges of $H_2$.
			By symmetry, it suffices to
			consider the following different cases:
			
			\begin{itemize}
			    \item [(a)] $e_1$ and $e'_1$ are adjacent;
			    
			    \item [(b)] $e_1$ and $e'_1$ are not adjacent but
			    lie on a common $4$-cycle;
			    
			    \item [(c)] $e_1$ and $e'_1$ are not adjacent and do not
			    lie on a common $4$-cycle.
			\end{itemize}
			
			If (a) holds, then $\varphi \in \mathcal{C}_4$.
			Suppose now that (b) holds. It suffices to prove
			that there are perfect matchings $M_1$ and $M_2$ 
			in $Q_d$, where $M_i$ contains all edges precolored $i$
		    and no other precolored edges, and where $M_1$ and $M_2$ satisfy that
		    the precolored edges of $Q_d-M_1\cup M_2$ lie in different
		    components.
		    We construct $M_1$
		    in the following way: include
		    $e_1$ and the unique non-precolored edge $e_3$ of $H_1$
		    that is in the same dimensional matching as $e_1$
		    and contained in a $4$-cycle with $e_1$; 
		    from $H_2$ we select the two edges corresponding
		    to the two opposite non-precolored edges of the $4$-cycle
		    containing $e_1$ and $e_3$;
		    for the remaining
		    edges of $M_1$ we choose four edges from $M$ that are 
		    adjacent to none of the edges $e_1$ and $e_3$.
		    We now define $M_2$ to consist of the the
		    edges from the unique perfect matching
		    in $H_1 - M_1$ containing $e'_1$ and of the edges from a perfect
		    matching of $H_2$ with no precolored edges.
		    
		    Suppose now that (c) holds. By symmetry, it suffices to
		    consider the two cases when $e_1$ and $e'_1$ are in the same dimensional
		    matching and when they are not. If the former holds, then we define
		    the matchings $M_1$ and $M_2$ exactly as in the preceding paragraph,
		    and it follows that $\varphi$ is extendable.
		    If $e_1$ and $e'_1$ are in different dimensional matchings,
		    then we select $M_1$ as the union of the dimensional matching
		    of $H_1$ containing $e_1$ and the unique dimensional matching
		    of $H_2$ with no precolored edge. As before, we can then
		    choose a perfect matching $M_2$ containing $e'_1$ and no
		    other precolored edges; the details are omitted.

			\bigskip
			
						\bigskip

		\noindent
		{\bf Case 3.2.} {\em 
		There is a
			precolored edge $e_1 =u_1v_1$ in $H_1$
			such that the corresponding edge of $H_2$ is not precolored:}
		
		\medskip
			
			Let 
			$e_2 = u_2v_2$ be the edge in $H_2$ corresponding to $e_1$.
			If some precolored edge of $H_2$ is not adjacent to $e_2$,
			then we may proceed as above:
			Assume without loss of generality
		    that such an edge is precolored $d$ in $H_2$, and that $\varphi(e_1)=1$.
		    Then we construct a new precoloring $\varphi'$
		    from $\varphi$ by coloring $e_2$ by color $d$, and
		    recoloring $e_1$ by color $d$. 
		    $H_1$ contains $d-2$ $\varphi'$-precolored edges,
		    so the restriction of $\varphi'$ to $H_1$ is extendable
		    by Theorem \ref{th:hypercube}. $H_2$ contains $3$ $\varphi'$-precolored
		    edges, so it is extendable unless $d=4$
		    and the restriction $\varphi_2$ of $\varphi'$
		    satisfies (C2) (with $d-1$ in place of $d$).
		    Assuming $d > 4$, we can choose these extensions so that they
		    use colors $2,\dots, d$, respectively, and we obtain
		    an extension of $\varphi$ by recoloring
		    $e_1$ and $e_2$ by color $1$, and thereafter coloring the edges of $M$.
		    If $d=4$, and $\varphi_2$ satisfies (C2), then $e_2$
		    and the two $\varphi$-precolored edges of $H_2$ form a matching, and none
		    of the $\varphi$-precolored edges in $H_2$ are adjacent to $e_2$.
		    It follows that for at least one of these two precolored
		    edges, the corresponding
		    edge in $H_1$ is not precolored; denote this edge by $e'_2$
		    and assume $\varphi(e'_2)=4$. 
		    Now, by Theorem \ref{th:hypercube},
		    the restriction of $\varphi$ to $H_1$ is extendable to a proper
		    edge coloring $f_1$ using colors $1,2,3$. Moreover, the
		    precoloring of $H_2$ obtained from the restriction of $\varphi$ by
		    recoloring $e'_2$ by the color of the corresponding edge $e'_1$ of $H_1$
		    under $f_1$ is, by Theorem \ref{th:hypercube}, extendable to a proper
		    edge coloring $f_2$ using colors $1,2,3$. By recoloring $e'_1$
		    and $e'_2$ by color $4$, and thereafter coloring the edges of $M$,
		    we obtain an extension of $\varphi$.

			Let us now assume that both precolored edges
			of $H_2$ are adjacent to $e_2$. In fact, we may assume that every precolored
			edge in $H_1$ either corresponds to a precolored edge of $H_2$ or
			is adjacent to both precolored edges of $H_2$. 
			Now, if both precolored edges
			of $H_2$ are incident with a common vertex $v$, then this implies
			that $\varphi \in \mathcal{C}$; so assume that $u_2$ is incident
			with one precolored edge and that $v_2$ is incident with one
			precolored edge.
			Clearly, this implies that at most $4$ edges are precolored in $H_1$,
			and thus $d \leq 6$.
			
			So let us assume that $d \leq 6$ and that
			$e_1=u_1v_1$ is precolored $1$. If the dimensional
			matching $M_1$ containing $e_1$ contains no other precolored edges,
			then the restriction $\varphi'$ of $\varphi$ to $Q_d -M_1$ is a
			precoloring of $d-1$ edges using $d-1$ colors. Furthermore,
			both components of $Q_d-M_1$ contain at most $d-2$
			precolored edges, so by Theorem \ref{th:hypercube},
			$\varphi'$ is extendable to to a proper edge coloring of $Q_d-M_1$
			using colors $2,\dots,d$. By coloring all edges of $M_1$ by color $1$,
			we obtain an extension of $\varphi$.
			
			If $M_1$ contains more than one precolored edge, then it contains exactly
			two precolored edges, $u_1v_1$ and $z_1t_1$, where $u_1v_1z_1t_1u_1$
			is a $4$-cycle in $H_1$. Let $u_2v_2z_2t_2u_2$ be the corresponding
			$4$-cycle of $H_2$, where $u_2t_2$ and $v_2z_2$ are precolored.
			We define $H_3$ to be the $3$-dimensional hypercube
			containing vertices $u_1,v_1, z_1, t_1, u_2,v_2, z_2, t_2$; then all
			precolored edges of $Q_d$ lie in $H_3$. Since $d \geq 4$, there
			is a dimensional matching $M_3$ in $Q_d$ which does not contain any edge
			from $H_3$.
			It follows that if $H'_1$ and $H'_2$ are the components of $Q_d-M_3$,
			then either $H'_1$ or $H'_2$ contains all precolored edges of $Q_d$;
			thus we may proceed as in the Case 1 above when $H_1$ contains exactly
			$d$ precolored edges.

					\bigskip
		
		\noindent
		{\bf Case 4.} {\em $H_1$ contains $d-1$ precolored edges and 
		$H_2$ contains $1$
		precolored edge:}
		
		\medskip
				
			Without loss of generality
			we assume that the edge in $H_2$ is precolored $d$.
			We first consider the case when the restriction of $\varphi$ to $H_1$ is
			extendable (as a precoloring of $Q_{d-1}$). Suppose first
			that there is some precolored edge $e_1$ in $H_1$
			such that the corresponding
			edge of $H_2$ is not precolored or adjacent to the precolored edge of $H_2$.
			Without loss of generality we assume that $\varphi(e_1) = 1$. 
			We define a new
			precoloring $\varphi'$ from $\varphi$ by recoloring $e_1$ 
			by color $d$ and by coloring
			$e_2$ by color $d$; this precoloring is proper, and, moreover, for $i=1,2$,
			the restriction of $\varphi'$ to $H_i$ is 
			extendable to a proper edge coloring $f_i$
			using colors $2,\dots,d$. By recoloring $e_1$ and $e_2$ by color $1$ and
			coloring every edge of $M$ by the color in 
			$\{1,\dots,d\}$ that is missing at
			its endpoints, we obtain an extension of $\varphi$.
			
			Suppose now that every precolored edge of $H_1$ either corresponds
			to a precolored edge of $H_2$, or 
			that the corresponding edge of $H_2$ is adjacent to
			a precolored edge of $H_2$.
			Since the restriction of $\varphi$ to $H_1$ is extendable, it follows
			that if $e_1 = u_1v_1$ is the edge of $H_1$ corresponding to the 
			precolored edge $e_2 =u_2v_2$ of
			$H_2$, then $e_1$ is precolored under $H_1$. Moreover, 
			since $\varphi \notin \mathcal{C}$,
			there are at least two precolored edges of $H_1$ incident with $u_1$ 
			and similarly for $v_1$. 
			Suppose
			e.g. that $\varphi(e_1)= 1$ and that color $2$ does 
			appear at $v_1$ under $\varphi$, but
			not at $u_1$, and that
			color $3$ appears at $u_1$.
			We define a new precoloring $\varphi'$ of $Q_d$ by recoloring
			the edge with color $3$ under $\varphi$ by color $2$,
			and by coloring the corresponding edge of $H_2$ by color $2$.
			Then, by Theorem \ref{th:hypercube},
			the restriction of $\varphi'$ to $H_2$ is extendable
			to a proper edge coloring using colors $1,2,4, \dots, d$,
			and the restriction of $\varphi'$ to $H_1$
			does not satisfy (C1), (C3) or (C4) (with $d-1$ in place of $d$). 
			Furthermore, since $2$ is the only color that appears on two edges
			under $\varphi'$, and these two edges are both adjacent to $e_1$,
			$\varphi'$ does not satisfy (C2). Hence,
			by the induction hypothesis,
			the restriction of $\varphi'$ to $H_1$
			is extendable to a proper edge coloring 
			$f_1$ using colors $1,2,4,\dots,d$.
			By recoloring the edges incident with $u_1$ and $u_2$ 
			with color $2$ by color
			$3$, we obtain proper edges colorings of $H_1$ and $H_2$, 
			such that we may color
			any edge of $M$ by the color missing at its endpoints
			to obtain an extension of $\varphi$.

			Let us now consider the case when the restriction 
			of $\varphi$ to $H_1$ is not extendable.
			Then there is some edge $u_1v_1$ in $H_1$ such that all precolored edges
			of $H_1$ are incident with $u_1$ or $v_1$ and $u_1v_1$ is not precolored.
			Without loss of generality, we assume that the 
			edge in $H_2$ is precolored $d$,
			there is some edge $e_1$ precolored $3$ incident with $u_1$ 
			such that the corresponding edge $e_2$ of $H_2$ is not precolored, 
			and there is an edge precolored $2$ incident with $v_1$.
			We define a new precoloring $\varphi'$ from $\varphi$ by recoloring $e_1$
			by color $2$ and coloring $e_2$ by color $2$.
			We may now finish the proof by proceeding exactly
			as in the preceding paragraph.
	\end{proof}
	
	This completes the proof of Theorem \ref{th:dprecol}.

\section{Concluding Remarks}

	In this paper we have obtained analogues for hypercubes of some
	classic results on completing partial Latin squares; 
	in general we believe that the
	following might be true. Here, $G^d$ denotes the $d$th power of
	the cartesian product of $G$ with itself.

\begin{conjecture}
\label{conj:upperbound}
	If $n$ and $d$ are positive integers,
	and $\varphi$ is a proper edge precoloring of $(K_{n,n})^d$
	with at most $nd-1$ precolored edges, then $\varphi$
	extends to a proper $nd$-edge coloring of $(K_{n,n})^d$.
\end{conjecture}

	Note that this
	is a generalization of both Evans' conjecture and the results
	obtained in this paper; Evans' conjecture is the case $d=1$, and
	the results obtained in this paper resolve the cases
	when $n=1$ and $n=2$; thus this conjecture is open 
	whenever $d \geq 2$ and $n \geq 3$.
		
	Given that a precoloring of at most $d-1$ precolored edges of $Q_d$ 
	or $K_{d,d}$ is
	always extendable, we might ask how many precolored edges
	of a general $d$-regular bipartite graph allow for an extension.
	Trivially, any precoloring of at most one edge of a graph $G$
	can be extended to a $\chi'(G)$-edge coloring of $G$.
	For larger sets of precolored edges, we have the following:

	\begin{proposition}
	\label{prop:dregular}
		For any $d \geq 2$, there is a $d$-regular bipartite graph
		with a precoloring $f$ of only $2$ edges, such that
		$f$ cannot be extended to a proper $d$-edge coloring.		
	\end{proposition}
	
	\begin{proof}
		Let $r> 1$ be a positive integer, and let $G_1,\dots, G_r$
		be $r$ copies of $K_{d,d} -e$, that is the complete bipartite
		graph with $d+d$ vertices with exactly one edge removed.
		From $G_1,\dots, G_r$ we form a $d$-regular graph $H$ by
		for $i=1,\dots,r$
		joining a vertex in $G_i$ of degree $d-1$ with 
		a vertex in $G_{i+1}$ of degree $d-1$ by an edge
		so that all added edges have distinct endpoints
		(indices taken modulo $r$).
		Let $e_1$ and $e_2$ be two distinct edges in $H$ joining
		vertices in distinct copies of $K_{d,d} -e$. We color
		$e_1$ with color $1$, and $e_2$ with color $2$. 
		Since any perfect matching in $H_1$ that contains $e_1$ also contains
		$e_2$, this precoloring cannot be extended
		to a proper $d$-edge coloring of $H$.
	\end{proof}
	
		Note that in the proof of Proposition \ref{prop:dregular},
		there is a similar precoloring with two edges colored $1$,
		which is not extendable to a proper $d$-edge coloring
		of the full graph. Also, the distance between the two 
		precolored edges
		can be made arbitrarily large.
		
		Furthermore, the examples given in the proof of 
		Proposition \ref{prop:dregular} 
		are $2$-connected.
		One may construct examples of 
		arbitrarily large connectivity by
		taking two copies $G_1$ and $G_2$ of $K_{n,n-1}$
		and for each vertex $v$ in $G_1$ of degree $n-1$ adding an edge between
		$v$ and its copy in $G_2$. The resulting graph 
		is $n$-regular, $(n-1)$-connected, 
		and the edge precoloring obtained by coloring any two edges with one
		endpoint in $G_1$ and one endpoint in $G_2$ by color $1$ is
		not extendable to a proper $d$-edge coloring of the full graph.

\section*{Acknowledgements}

We wish to thank the referees for useful comments, and particularly one 
of the referees for a very careful reading that identified
several mistakes and simplified some of the arguments.


\end{document}